\documentclass[12pt,a4paper,reqno]{amsart}
\usepackage[a4paper,left=16mm,right=16mm,top=36mm,bottom=36mm]{geometry}

\usepackage{amssymb}
\usepackage{amsmath}
\usepackage{graphicx}
\usepackage{float}
\usepackage[utf8]{inputenc}
\usepackage[english]{babel}
\usepackage{stmaryrd}
\usepackage{mathtools}
\usepackage{subcaption}
\usepackage{pgfplots}
\usepackage{grffile}
\usepackage{bbm}
\pgfplotsset{compat=newest}
\usetikzlibrary{plotmarks}
\usetikzlibrary{arrows.meta}
\usepgfplotslibrary{patchplots}

\definecolor{mycolor1}{rgb}{0.46600,0.67400,0.18800}%
\definecolor{mycolor2}{rgb}{0.30100,0.74500,0.93300}%

\DeclarePairedDelimiter\floor{\lfloor}{\rfloor}
\newcommand{\R}{\mathbb{R}}
\newcommand{\N}{\mathbb{N}}
\newcommand{\eps}{\varepsilon}
\newcommand{\fhi}{\varphi}

\newcommand{\vertiii}[1]{{\left\vert\kern-0.25ex\left\vert\kern-0.25ex\left\vert #1
    \right\vert\kern-0.25ex\right\vert\kern-0.25ex\right\vert}}
\newcommand{\supp}{\mathrm{supp}}
\def\calA{\mathcal{A}}
\def\calC{\mathcal{C}}

\def\calK{\mathcal{K}}
\def\calO{\mathcal{O}}

\def\calN{\mathcal{N}}
\def\calT{\mathcal{T}}

\def\calF{\mathcal{F}}
\def\calP{\mathcal{P}}

\def\Xint#1{\mathchoice
   {\XXint\displaystyle\textstyle{#1}}%
   {\XXint\textstyle\scriptstyle{#1}}%
   {\XXint\scriptstyle\scriptscriptstyle{#1}}%
   {\XXint\scriptscriptstyle\scriptscriptstyle{#1}}%
   \!\int}
\def\XXint#1#2#3{{\setbox0=\hbox{$#1{#2#3}{\int}$}
     \vcenter{\hbox{$#2#3$}}\kern-.5\wd0}}
\def\meanint{\Xint-}

\numberwithin{equation}{section}

\usepackage[]{hyperref}
\usepackage[noabbrev, capitalise, nameinlink]{cleveref}
\crefname{equation}{}{}
\crefname{assumption}{Assumption}{Assumptions}
\crefname{rmrk}{Remark}{Remarks}
\crefformat{equation}{\textup{#2(#1)#3}}

\usepackage[colorinlistoftodos,prependcaption]{todonotes}

\newtheorem{rmrk}{Remark}[section]    
\newtheorem{thrm}{Theorem}[section]     
\newtheorem{lmm}{Lemma}[section]

\newcommand{\sign}{\operatorname{sign}}

\begin{document}

\title[Computational multiscale methods for nondivergence-form PDEs]{Computational multiscale methods for nondivergence-form elliptic partial differential equations}

\author[P. Freese]{Philip Freese}
\address[Philip Freese]{Technische Universit\"{a}t Hamburg, Institut f\"{u}r Mathematik, Am Schwarzenberg-Campus 3, 21073 Hamburg, Germany.}
\email{philip.freese@tuhh.de}

\author[D. Gallistl]{Dietmar Gallistl}
\address[Dietmar Gallistl]{Friedrich-Schiller-Universit\"{a}t Jena, Institut f\"{u}r Mathematik, Ernst-Abbe-Platz 2, 07743 Jena, Germany.}
\email{dietmar.gallistl@uni-jena.de}

\author[D. Peterseim]{Daniel Peterseim}
\address[Daniel Peterseim]{Universit\"{a}t Augsburg, Institut f\"{u}r Mathematik \& Centre for Advanced Analytics and Predictive Sciences (CAAPS), Universit\"atsstr.~12a, 86159 Augsburg, Germany.}
\email{daniel.peterseim@uni-a.de}

\author[T. Sprekeler]{Timo Sprekeler}
\address[Timo Sprekeler]{National University of Singapore, Department of Mathematics, 10 Lower Kent Ridge Road, Singapore 119076, Singapore.}
\email{timo.sprekeler@nus.edu.sg}

\subjclass[2010]{35J15, 65N12, 65N30}
\keywords{nondivergence-form elliptic PDE; localized orthogonal decomposition; numerical homogenization; finite element methods}
\date{\today}

\begin{abstract}
This paper proposes novel computational multiscale methods for linear second-order elliptic partial differential equations in nondivergence-form with heterogeneous coefficients satisfying a Cordes condition. The construction follows the methodology of localized orthogonal decomposition (LOD) and provides operator-adapted coarse spaces by solving localized cell problems on a fine scale in the spirit of numerical homogenization. The degrees of freedom of the coarse spaces are related to nonconforming and mixed finite element methods for homogeneous problems. The rigorous error analysis of one exemplary approach shows that the favorable properties of the LOD methodology known from divergence-form PDEs, i.e., its applicability and accuracy beyond scale separation and periodicity, remain valid for problems in nondivergence-form. 
\end{abstract}

\maketitle

\section{Introduction}

In this work, we consider linear second-order elliptic partial differential equations of the form
\begin{align}\label{intro1}
A:D^2 u + b\cdot \nabla u - c\, u \coloneqq  \sum_{i,j = 1}^n a_{ij}\, \partial_{ij}^2 u + \sum_{k = 1}^n b_k\, \partial_k u - c\, u = f\quad\text{in }\Omega,
\end{align}
posed on a bounded convex polyhedral domain $\Omega\subset \R^n$, $n\geq 2$, with a right-hand side $f\in L^2(\Omega)$, subject to the homogeneous Dirichlet boundary condition
\begin{align}\label{intro2}
u = 0\quad\text{on }\partial\Omega,
\end{align}
where $A = (a_{ij})_{1\leq i,j\leq n}\in L^{\infty}(\Omega;\R^{n\times n}_{\mathrm{sym}})$, $b = (b_k)_{1\leq k\leq n}\in L^{\infty}(\Omega;\R^n)$, and $c\in L^{\infty}(\Omega)$ are heterogeneous coefficients such that $A$ is uniformly elliptic, $c\geq 0$ almost everywhere in $\Omega$, and the triple $(A,b,c)$ satisfies a (generalized) Cordes condition. Our main objective in this paper is to propose and rigorously analyze a novel finite element scheme for the accurate numerical approximation of the solution to the multiscale problem \cref{intro1}--\cref{intro2}, a task we will refer to as numerical homogenization, by following the methodology of localized orthogonal decomposition (LOD) \cite{MaP14,MalP20}. It is worth mentioning that we are working in a framework beyond periodicity and separation of scales.   

The motivation for investigating \cref{intro1}--\cref{intro2} stems from engineering, physics, and mathematical areas such as stochastic analysis. Notably, such equations arise in the linearization of Hamilton--Jacobi--Bellman (HJB) equations from stochastic control theory. A distinguishing feature of nondivergence-form problems such as \cref{intro1}--\cref{intro2} is the absence of a natural variational formulation. However, due to the Cordes condition, there exists a unique strong solution to \cref{intro1}--\cref{intro2} which can be equivalently characterized as the unique solution to the Lax--Milgram-type problem
\begin{align}\label{intro3}
\text{seek }u\in V \coloneqq  H^2(\Omega)\cap H^1_0(\Omega)\text{ such that }a(u,v) = \langle F,v\rangle\text{ for all }v\in V
\end{align}
with some suitably defined $F\in V^{\ast}$ and bounded coercive bilinear form $a:V\times V\rightarrow \R$. 

In the presence of coefficients that vary on a fine scale, e.g., when $A(x) = \tilde{A}(\frac{x}{\eps})$ with some $(0,1)^n$-periodic $\tilde{A}\in L^{\infty}(\R^n;\R^{n\times n}_{\mathrm{sym}})$ and $\eps > 0$ small, classical finite element methods are being outperformed by multiscale finite element methods such as developed in this paper. For periodic coefficients, periodic homogenization has been proposed for linear elliptic equations in nondivergence-form, cf. \cite{AL89,BLP11,GST22,GTY20,JKO94,KL16,Spr23,ST21}. Numerical homogenization of such problems has not been studied extensively so far. A finite element numerical homogenization scheme for the periodic setting has been proposed and analyzed in \cite{CSS20}, which is based on an approximation of the solution to the homogenized problem via a finite element approximation of an invariant measure (see also \cite{Spr23}). Further, there has been some previous study on finite difference approaches for such problems in the periodic setting; see \cite{AK18,FO09}. Concerning fully-nonlinear HJB and Isaacs equations, finite element approaches for the numerical homogenization in the periodic setting have been suggested in \cite{GSS21,KS22} and some finite difference schemes have been studied in \cite{CM09,FO18,FOO18}. 

The case of arbitrarily rough coefficients beyond periodicity and scale separation has not yet been addressed. For divergence-form PDEs, several numerical homogenization methods have been developed in the last decade, which are based on the construction of operator-adapted basis functions and are applicable without such structural assumptions. We highlight the LOD \cite{MaP14,HeP13,KPY18,Mai20}, the Generalized Finite Element Method \cite{BaL11,EGH13,Ma22}, Rough Polyharmonic Splines and Gamblets \cite{OZB14,Owh17} as well as the recently proposed Super-Localized Orthogonal Decomposition \cite{HP21pre,FHP21,BFP22}. 

The aim of this paper is to transfer such modern numerical homogenization methods to the case of nondivergence-form problems, and to provide a proof of concept that this framework also applies to this class of equations. The only existing link between numerical homogenization and nondivergence-form problems is the metric-based upscaling proposed in \cite{OZ07} which exploits nondivergence-form problems for a problem-dependent change of metric as part of the numerical homogenization of divergence-form problems. Our construction of a practical finite element method for the nondivergence-form problem \cref{intro1}--\cref{intro2} in the presence of multiscale data follows the abstract LOD framework for numerical homogenization methods for divergence-form problems presented in \cite{AHP21}. It is based on the problem \cref{intro3} as starting point, $a$-orthogonal decompositions of the solution space $V$ and the test space $V$ into a fine-scale space -- defined as the intersection of the kernels of suitably chosen quantities of interest $q_1,\dots,q_N\in V^{\ast}$ -- and some coarse scale space, and a localization argument. In our exemplary approach, the choice of quantities of interests is inspired by the degrees of freedom of the nonconforming Morley finite element. 

The remainder of this work is organized as follows. In \cref{sec:Problem Setting and Well-Posedness}, we present the problem setting as well as the theoretical foundation including the well-posedness of \cref{intro1}--\cref{intro2} based on a Cordes condition. In \Cref{sec:Numerical homogenization scheme}, we introduce the numerical homogenization scheme for the approximation of the solution to \cref{intro1}--\cref{intro2} based on LOD theory. The proposed numerical homogenization scheme is rigorously analyzed and error bounds are proved. The numerical implementation is based on a $H^2$-conforming Birkhoff--Mansfield element and is introduced in \cref{sec:Fine-scale discretization by conforming Birkhoff--Mansfield elements}. In \cref{sec:Numerical Experiments}, we illustrate the theoretical findings by several numerical experiments and finally, in \cref{sec:Alternative discretization using a mixed non-symmetric FEM}, we discuss an alternative discretization based on mixed finite element theory.

\section{Problem setting and well-posedness}\label{sec:Problem Setting and Well-Posedness}

\subsection{Framework}

For a bounded convex polyhedral domain $\Omega\subset \R^n$ in dimension $n\geq 2$, and a right-hand side $f\in L^2(\Omega)$, we consider the problem
\begin{align}\label{u prob}
\left\{ \begin{aligned}
Lu \coloneqq A:D^2 u + b\cdot \nabla u - c\, u &= f\quad\text{in }\Omega,\\
u &= 0 \quad\text{on }\partial\Omega,
\end{aligned}\right.
\end{align}
where we assume that
\begin{align*}
A\in L^{\infty}(\Omega;\R^{n\times n}_{\mathrm{sym}}),\quad b\in L^{\infty}(\Omega;\R^{n}),\quad c\in L^{\infty}(\Omega)\text{ with }c\geq 0 \text{ a.e. in }\Omega,
\end{align*}
that $A$ is uniformly elliptic, i.e., there exist constants $\zeta_1,\zeta_2 > 0$ such that
\begin{align}\label{unifm ell}
\forall \xi\in \R^n\backslash\{0\}:\quad \zeta_1\leq  \frac{A\xi\cdot \xi}{\lvert \xi\rvert^2} \leq \zeta_2 \quad \text{a.e. in }\Omega,
\end{align}
and that the triple $(A,b,c)$ satisfies the Cordes condition, that is, we make the following assumption:
\begin{itemize}
\item[(C1)] If $\lvert b\rvert = c = 0$ a.e. in $\Omega$, we assume that there exists a constant $\delta\in (0,1)$ such that
\begin{align}\label{CordesC1}
\frac{\lvert A \rvert^2}{(\mathrm{tr}(A))^2}\leq \frac{1}{n-1+\delta}\quad\text{ a.e. in }\Omega.
\end{align}
Further, in this case we set $\gamma\coloneqq \frac{\mathrm{tr}(A)}{\lvert A \rvert^2}$ and $\lambda \coloneqq  0$.

\item[(C2)] Otherwise, we assume that there exist constants $\delta\in (0,1)$ and $\lambda \in (0,\infty)$ such that
\begin{align}\label{CordesC2}
\frac{\lvert A\rvert^2 + \frac{1}{2\lambda}\lvert b\rvert^2 + \frac{1}{\lambda^2}c^2}{\left(\mathrm{tr}(A) + \frac{1}{\lambda}c\right)^2}\leq \frac{1}{n + \delta}\quad\text{ a.e. in }\Omega.
\end{align}
Further, in this case we set $\gamma \coloneqq  \frac{\mathrm{tr}(A) + \frac{1}{\lambda}c}{\lvert A\rvert^2 + \frac{1}{2\lambda}\lvert b\rvert^2 + \frac{1}{\lambda^2}c^2}$.
\end{itemize}
Here, we have used the notation $\lvert M\rvert \coloneqq  \sqrt{M:M}$ to denote the Frobenius norm of $M\in \R^{n\times n}$.

\begin{rmrk}\label{Rk: C1 in 2d}
When $n = 2$, uniform ellipticity \cref{unifm ell} guarantees that the condition \cref{CordesC1} is satisfied for some $\delta\in (0,1)$; see e.g., \cite{SS14}.   
\end{rmrk}

\begin{rmrk}[Properties of $\gamma$]
Note that $\gamma\in L^{\infty}(\Omega)$ and that there exist constants $\gamma_0,\Gamma > 0$ depending only on $n,\zeta_1,\zeta_2,\lambda,\|b\|_{L^{\infty}(\Omega)},\|c\|_{L^{\infty}(\Omega)}$ such that $\gamma_0 \leq \gamma \leq \Gamma$ a.e. in $\Omega$.
\end{rmrk}

\subsection{Well-posedness}

We introduce the Hilbert space $(V,(\cdot,\cdot)_V)$ by setting
\begin{align}\label{V defn}
V\coloneqq H^2(\Omega)\cap H^1_0(\Omega),\qquad (\cdot,\cdot)_V \coloneqq  (\cdot,\cdot)_{H^2(\Omega)},
\end{align}
and we write $\|\cdot\|_V\coloneqq \|\cdot\|_{H^2(\Omega)}$ and $\|\cdot\|_{V,\omega}\coloneqq \|\cdot\|_{H^2(\omega)}$ for any subdomain $\omega\subset \Omega$. Then, introducing the bilinear form
\begin{align*}
a:V\times V \rightarrow \R,\qquad  a(v_1,v_2)\coloneqq  \left( \gamma\, L v_1,\Delta v_2 - \lambda v_2 \right)_{L^2(\Omega)},
\end{align*}
the linear operator
\begin{align*}
\calA:V\rightarrow V^{\ast},\qquad v\mapsto \calA v\coloneqq  a(v,\cdot),
\end{align*}
and the linear functional
\begin{align*}
F:V\rightarrow \R,\qquad  v\mapsto \langle F,v\rangle \coloneqq  \left( \gamma f,\Delta v - \lambda v \right)_{L^2(\Omega)},
\end{align*}
it is well-known that we have existence and uniqueness of a strong solution to \eqref{u prob}; see \cite{SS13,SS14}:
\begin{thrm}[Well-posedness]\label{Thm: Well-posed u problem}
The following assertions hold true.
\begin{itemize}
\item[(i)] A function $u\in V$ is a strong solution to \cref{u prob} if, and only if, 
\begin{align}
	a(u,v) = \langle F,v\rangle\quad\text{for all }v\in V.
	\label{eq:weak}
\end{align}
\item[(ii)] There exists a unique $u\in V$ such that $\calA u = F$ in $V^{\ast}$, i.e., $a(u,v) = \langle F,v\rangle$ for all $v\in V$. 
\end{itemize}
In particular, the problem \cref{u prob} has a unique strong solution $u\in V$.
\end{thrm}
Note that assertion (i) of \cref{Thm: Well-posed u problem} follows immediately from the fact that for any $g\in L^2(\Omega)$ there exists a unique $v\in V$ such that $\Delta v - \lambda v = g$, and the positivity of the renormalization function $\gamma$. Assertion (ii) of \cref{Thm: Well-posed u problem} is shown by a standard Lax--Milgram argument using the properties of $a$ and $F$ listed below.
\begin{lmm}[Properties of the maps $a$ and $F$]\label{Lmm: prop of a F}
The following assertions hold true.
\begin{itemize}
\item[(i)] Local boundedness of $a$: There exists a constant $C_a >0$ depending only on $n,\zeta_1,\zeta_2$, $\lambda$,$\|b\|_{L^{\infty}(\Omega)}$,$\|c\|_{L^{\infty}(\Omega)}$ such that for any subdomains $\omega_1,\omega_2\subset \Omega$ we have that
\begin{align*}
\forall v_1,v_2\in V:\quad \supp(v_1)\subset \omega_1,\supp(v_2)\subset \omega_2\quad\Longrightarrow\quad
\lvert a(v_1,v_2)\rvert \leq C_a \|v_1\|_{V,\omega_1\cap \omega_2} \|v_2\|_{V,\omega_1\cap \omega_2}.
\end{align*}
\item[(ii)] Coercivity of $a$: There exists a constant $\alpha > 0$ depending only on $\mathrm{diam}(\Omega),n,\delta$ such that
\begin{align*}
a(v,v)\geq (1-\sqrt{1-\delta})\|\Delta v - \lambda v\|_{L^2(\Omega)}^2 \geq  \alpha \|v\|_V^2\qquad \forall v\in V.
\end{align*}
\item[(iii)] $F\in V^*$: There exists a constant $\mu >0$ depending only on $n,\zeta_1,\zeta_2,\lambda,\|b\|_{L^{\infty}(\Omega)},\|c\|_{L^{\infty}(\Omega)}$ such that
\begin{align*}
\lvert \langle F,v\rangle\rvert \leq \mu \|f\|_{L^2(\Omega)} \|v\|_V\qquad \forall v\in V,
\end{align*}
or equivalently, $\|F\|_{V^*}\leq \mu \|f\|_{L^2(\Omega)}$.
\end{itemize}
\end{lmm}

The proofs of assertions (i) and (iii) of \cref{Lmm: prop of a F} are straightforward. A proof of assertion (ii) of \cref{Lmm: prop of a F} can be found in \cite{SS13,SS14}, relying on the observation that the Cordes condition implies that for any subdomain $\omega\subset\Omega$ we have that (see Lemma 1 in \cite{SS14})
\begin{align*}
\forall v\in H^2(\omega):\qquad\lvert \gamma\, Lv-(\Delta v-\lambda v)\rvert \leq \sqrt{1-\delta}\sqrt{\lvert D^2 v\rvert^2 + 2\lambda \lvert \nabla v\rvert^2 + \lambda^2 v^2  }\quad\text{a.e. in }\omega,
\end{align*}
and using the Miranda--Talenti-type estimates (see Theorem 2 in \cite{SS14})
\begin{align*}
\|D^2 v\|_{L^2(\Omega)}^2 + 2\lambda \|\nabla v\|_{L^2(\Omega)}^2 + \lambda^2 \|v\|_{L^2(\Omega)}^2 &\leq \|\Delta v - \lambda v\|_{L^2(\Omega)}^2\qquad\quad\;\;\; \forall v\in V, \\ \|v\|_{V}&\leq C_{\mathrm{MT}}\|\Delta v-\lambda v\|_{L^2(\Omega)}\qquad \forall v\in V,
\end{align*}
with a constant $C_{\mathrm{MT}}>0$ depending only on $\mathrm{diam}(\Omega)$ and $n$.  

\begin{rmrk}\label{Rk: H2 bound on u}
In view of \cref{Thm: Well-posed u problem,Lmm: prop of a F}, the unique strong solution $u\in V$ to \cref{u prob} satisfies the bound
\begin{align*}
\|u\|_V \leq \alpha^{-1}\|F\|_{V^*}\leq   \alpha^{-1}\mu\|f\|_{L^2(\Omega)},
\end{align*}
where $\alpha,\mu> 0$ are the constants from \cref{Lmm: prop of a F}(ii)--(iii).
\end{rmrk}

It is worth emphasizing that in the setting of periodic homogenization, i.e., $A = \tilde{A}(\frac{\cdot}{\eps})$, $b = \tilde{b}(\frac{\cdot}{\eps})$, $c = \tilde{c}(\frac{\cdot}{\eps})$ for some small parameter $\varepsilon>0$ and $(0,1)^n$-periodic $\tilde{A},\tilde{b},\tilde{c}\in L^{\infty}(\R^n)$ satisfying the Cordes condition in $\R^n$, we have that the $H^2(\Omega)$-norm of the solution to \eqref{u prob} is uniformly bounded in $\eps\in (0,1)$, while generically the $H^{2+s}(\Omega)$-norm is unbounded as $\eps\searrow 0$ for any $s>0$. Note that this is different to the usual periodic homogenization setting for divergence-form equations where generically the $H^2(\Omega)$-norm is unbounded as $\eps\searrow 0$.

\section{Numerical homogenization scheme}\label{sec:Numerical homogenization scheme}

For simplicity, we only work in dimension $n = 2$ and give some remarks on numerical homogenization in higher dimensions in \cref{Sec: ext and fw}. 

\subsection{Fine-scale space}

We start by introducing a triangulation of the bounded convex polygon $\Omega\subset \R^2$. Thereafter, we define a certain closed linear subspace $W$ of $V$ (recall the definition of $V$ from \cref{V defn}) which will be referred to as the fine-scale space. 

\subsubsection{Triangulation}

Let $\calT_H$ be a regular quasi-uniform triangulation of $\Omega$ into closed triangles with mesh-size $H>0$ and shape-regularity parameter $\rho>0$ given by
\begin{align*}
H \coloneqq  \max_{T\in \calT_H} \mathrm{diam}(T),\qquad \rho \coloneqq  H^{-1} \min_{T\in \calT_H} \rho_T, 
\end{align*}
where $\rho_T$ denotes the diameter of the largest ball which can be inscribed in the element $T$. We introduce the piecewise constant mesh-size function $h_{\calT_H}:\bar{\Omega}\rightarrow \R$ given by $\left.h_{\calT_H}\right\rvert_T \coloneqq  H_T \coloneqq  \lvert T\rvert^{\frac{1}{2}}$ for $T\in \calT_H$. Let $\calF_H$ denote the set of edges, $\calN_H^{\mathrm{int}}$ the set of interior vertices, $\calN_H^{\partial}$ the set of boundary vertices, and define 
\begin{align*}
N_1 \coloneqq  \lvert \calF_H\rvert,\qquad N_2\coloneqq \lvert \calN_H^{\mathrm{int}}\rvert,\qquad N\coloneqq  N_1 + N_2.
\end{align*}  
We label the edges $F_1,\dots,F_{N_1}$ and the interior vertices $z_1,\dots,z_{N_2}$, so that 
\begin{align*}
\calF_{H} = \{F_1,F_2,\dots,F_{N_{1}}\},\qquad \calN_H^{\mathrm{int}} = \{z_1,z_2,\dots, z_{N_2}\}.
\end{align*}
For each edge $F\in \calF_H$ we associate a fixed choice of unit normal $\nu_F$, where we often drop the subscript and only write $\nu$ for simplicity. Finally, for a subset $S\subset \Omega$, we define $N^0(S)\coloneqq S$ and $N^\ell(S) \coloneqq  \bigcup \{T\in \calT_H:T\cap N^{\ell-1}(S)\neq \emptyset\}$ for $\ell\in\N$.

\subsubsection{Quantities of interest and the space of fine-scale functions}

First, let us note that $V\subset C(\bar{\Omega})$ by Sobolev embeddings. For $i\in \{1,\dots,N\}$, we define the quantity of interest $q_i\in V^{\ast}$ by
\begin{align*}
q_i:V\rightarrow \R,\qquad v\mapsto \langle q_i,v\rangle \coloneqq  \begin{cases}
\meanint_{F_i} \nabla v\cdot \nu&,\text{ if }1\leq i\leq N_{1},\\
v(z_{i-N_{1}})&,\text{ if }N_{1}+1 \leq i\leq  N.
\end{cases}
\end{align*}
The quantities of interest $\{q_1,\dots,q_N\}\subset V^{\ast}$ are linearly independent as can be seen from the fact that there exist functions $u_1,\dots,u_N\in V$ such that $\langle q_i,u_j\rangle = \delta_{ij}$ for any $i,j\in\{1,\dots,N\}$; see \cref{Sec: constr locbas}. We define the fine-scale space
\begin{align}\label{W defn}
W\coloneqq  \bigcap_{i\in \{1,\dots,N\}} \mathrm{ker}(q_i) = \{\left.v\in V\,\right\rvert \langle q_{i},v\rangle = 0\;\;\forall\,i\in\{1,\dots,N\}\},
\end{align}
which is a closed linear subspace of $V$. 

\subsubsection{Connection to the Morley finite element space}\label{Sec: Conn to Mor}
We consider the Morley finite element space
\begin{align*}
V_H^{\mathrm{Mor}} \coloneqq  \{\left.v\in \calP_2(\calT_H)\,\right\rvert &v \text{ is continuous at }\calN_H^{\mathrm{int}} \text{ and vanishes at }\calN_H^{\partial}; \\& D_{\mathrm{NC}}v\text{ is continuous at the interior edges' midpoints} \},
\end{align*}
whose local degrees of freedom are the evaluation of the function at each vertex and the evaluation of the normal derivative at the edges' midpoints. Here, the piecewise action of the differential operator $D$ is indicated by the subscript $\mathrm{NC}$, i.e., we define $\left.( D_{\mathrm{NC}}v)\right\rvert_T := \left.D(v\right\rvert_T)$ for any $T\in \calT_H$. Then, letting $\{\phi_1,\dots,\phi_N\}\subset V_H^{\mathrm{Mor}}$ denote Morley basis functions satisfying $\langle q_i,\phi_j\rangle = \delta_{ij}$ for all $i,j\in\{1,\dots,N\}$ (note $\langle q_i,\phi_j\rangle$ is well-defined although $\phi_j\not\in V$), we have that the Morley interpolation operator is given by
\begin{align}\label{PiMo}
\Pi^{\mathrm{Mor}}:V\rightarrow V_H^{\mathrm{Mor}},\qquad v\mapsto \Pi^{\mathrm{Mor}}v\coloneqq  \sum_{i=1}^N \langle q_i,v\rangle \phi_i, 
\end{align}
and we observe that
\begin{align*}
W = \mathrm{ker}(\Pi^{\mathrm{Mor}}) = \{\left.v\in V\,\right\rvert \Pi^{\mathrm{Mor}}v = 0\}.
\end{align*}
In particular, using Morley interpolation bounds (see e.g., \cite{WX13}), we have the local estimate
\begin{align}\label{bd for w in W loc}
\|w\|_{L^2(T)} + H_T\|\nabla w\|_{L^2(T)} \lesssim H_T^2 \|D^2 w\|_{L^2(T)}\qquad\forall T\in \calT_H\quad \forall w\in W,
\end{align}
and the global bound
\begin{align}\label{bd for w in W}
\|w\|_{L^2(\Omega)} + H\|\nabla w\|_{L^2(\Omega)}  \lesssim H^2 \|D^2 w\|_{L^2(\Omega)}\qquad \forall w\in W.
\end{align}

\subsubsection{Projectors onto the fine-scale space}\label{Subsubsec: C Cstar}

We introduce the maps
\begin{align*}
\calC:V&\rightarrow W,\qquad v\mapsto \calC v,\\ \calC^{\ast}:V&\rightarrow W,\qquad v\mapsto \calC^{\ast} v,
\end{align*}
where for $v\in V$ we define $\calC v\in W$ to be the unique function in $W$ that satisfies
\begin{align*}
a(\calC v,w) = a(v,w)\qquad \forall w\in W,
\end{align*} 
and we define $\calC^{\ast}v\in W$ to be the unique function in $W$ that satisfies
\begin{align*}
a(w,\calC^{\ast}v) = a(w,v)\qquad \forall w\in W.
\end{align*} 
\begin{rmrk}\label{Rk: C well def}
In view of \cref{Lmm: prop of a F}, we have by the Lax--Milgram theorem that the maps $\calC,\calC^{\ast}$ are well-defined, and we have the bounds 
\begin{align*}
\|\calC v\|_V \leq \alpha^{-1}C_a \|v\|_V,\qquad \|\calC^{\ast} v\|_V \leq \alpha^{-1} C_a \|v\|_V\qquad \forall v\in V.
\end{align*}
Further, the maps $\calC,\calC^{\ast}$ are surjective and continuous projectors onto $W$, and we have that 
\begin{align*}
W = \mathrm{ker}(1-\calC) = \mathrm{ker}(1-\calC^{\ast}).
\end{align*}
\end{rmrk}

\subsection{Ideal numerical homogenization scheme}

\subsubsection{$a$-orthogonal decompositions of $V$}

We define the trial space $\tilde{U}_H\subset V$ and the test space $\tilde{V}_H\subset V$ by
\begin{align}\label{UHtilde}
\tilde{U}_H \coloneqq  (1-\calC)V,\qquad \tilde{V}_H \coloneqq  (1-\calC^{\ast})V.
\end{align}
In view of \cref{Rk: C well def}, we then have the following decompositions of the space $V$:
\begin{align}\label{a ort dec}
V = (1-\calC)V\oplus \calC V = \tilde{U}_H \oplus W,\qquad
V = (1-\calC^{\ast})V\oplus \calC^{\ast} V = \tilde{V}_H\oplus W. 
\end{align}
We state a few observations below.
\begin{lmm}[Properties of $\tilde{U}_H$ and $\tilde{V}_H$]\label{Lmm: Prop UHtilde} The following assertions hold true.
\begin{itemize}
\item[(i)] We have that $\dim(\tilde{U}_H) = \dim(\tilde{V}_H) = N$.
\item[(ii)] The decompositions \cref{a ort dec}  are $a$-orthogonal in the sense that $a(\tilde{U}_H,W) = 0$ and $a(W,\tilde{V}_H) = 0$.
\item[(iii)] We can equivalently characterize the spaces $\tilde{U}_H$ and $\tilde{V}_H$ via
\begin{align*}
\tilde{U}_H &= \{\left.v\in V \,\right\rvert a(v,w) = 0\;\forall w\in W \},\\
\tilde{V}_H &= \{\left.v\in V \,\right\rvert a(w,v) = 0\;\forall w\in W \}.
\end{align*}
\item[(iv)] We have that $\tilde{U}_H = \mathrm{span}(\calA^{-1} q_{1},\dots,\calA^{-1} q_{N})$.
\end{itemize}
\end{lmm}

\begin{proof}
\begin{itemize}
\item[(i)] By the Riesz representation theorem, there exist $\hat{q}_1,\dots,\hat{q}_{N}\in V$ such that $q_{i} = (\,\cdot\,,\hat{q}_{i})_V$ in $V^{\ast}$ for $i\in\{1,\dots,N\}$. Set $S\coloneqq \mathrm{span}(\hat{q}_{1},\dots,\hat{q}_{N})$ and note that $\dim(S) = N$ as the quantities of interest $q_1,\dots, q_N$ are linearly independent. Then, in view of \cref{W defn}, we have that $W = S^{\perp}$ and there holds $V = W^{\perp} \oplus W =  S \oplus W$. The claim follows.
\item[(ii)] This follows immediately from the definition of the spaces $\tilde{U}_H$ and $\tilde{V}_H$ from \cref{UHtilde}, and the definitions of the projectors $\calC$ and $\calC^{\ast}$ from \cref{Subsubsec: C Cstar}. 
\item[(iii)] By the properties of the projectors $\calC$ and $\calC^{\ast}$ from \cref{Subsubsec: C Cstar}, we have that 
\begin{align*}
\{\left.v\in V \,\right\rvert a(v,w) = 0\;\forall w\in W \} &= \{\left.v\in V \,\right\rvert \calC v = 0 \} = (1-\calC)V = \tilde{U}_H,\\
\{\left.v\in V \,\right\rvert a(w,v) = 0\;\forall w\in W \} &= \{\left.v\in V \,\right\rvert \calC^{\ast} v = 0 \} = (1-\calC^{\ast})V= \tilde{V}_H.
\end{align*}
\item[(iv)] First, note that $\dim(\tilde{U}_H) = N =\dim(\mathrm{span}(\calA^{-1} q_{1},\dots,\calA^{-1} q_{N}))$. Next, we observe that for $i\in\{1,\dots,N\}$ we have that  $a(\calA^{-1} q_i,w) = \langle q_i,w\rangle = 0 $ for all $w\in W$, i.e., there holds $\calA^{-1} q_{1},\dots,\calA^{-1} q_{N} \in \tilde{U}_H$. It follows that $\tilde{U}_H = \mathrm{span}(\calA^{-1} q_{1},\dots,\calA^{-1} q_{N})$.
\end{itemize}
\end{proof}

\subsubsection{Ideal numerical homogenization}

The ideal discrete problem is the following:
\begin{align}\label{var problem discrete}
\text{Find }\tilde{u}_H\in \tilde{U}_H\text{ such that }\quad a(\tilde{u}_H,\tilde{v}_H) = \langle F,\tilde{v}_H\rangle  \quad \forall \tilde{v}_H\in \tilde{V}_H.
\end{align}
\begin{thrm}[Analysis of the ideal discrete problem]\label{Thm: Analysis ideal method}
There exists a unique solution $\tilde{u}_H\in \tilde{U}_H$ to the ideal discrete problem \cref{var problem discrete}. Further, denoting the unique strong solution to \cref{u prob} by $u\in V$, the following assertions hold true.
\begin{itemize}
\item[(i)] We have the bound \begin{align*}
\|\tilde{u}_H\|_V \leq \alpha^{-2} C_a \|F\|_{(\tilde{V}_H)^{\ast}} \leq \alpha^{-2}\mu\, C_a  \|f\|_{L^2(\Omega)}.
\end{align*}
\item[(ii)] We have that $u-\tilde{u}_H = \calC u\in W$ and hence, \begin{align*}
\langle q_i,\tilde{u}_H\rangle = \langle q_i,u\rangle\qquad \forall i\in \{1,\dots,N\},
\end{align*}
i.e., the quantities of interest are conserved.
\item[(iii)] We have the error bound
\begin{align*}
\|u-\tilde{u}_H\|_{L^2(\Omega)} + H\|\nabla(u-\tilde{u}_H)\|_{L^2(\Omega)}\lesssim H^2 \|f\|_{L^2(\Omega)}
\end{align*}
for the approximation of the true solution $u\in V$ by $\tilde{u}_H\in \tilde{U}_H$. 
\end{itemize}

\end{thrm}
\begin{proof}
First, recall the properties of $a$ and $F$ from \cref{Lmm: prop of a F}. Next, we note that for any $\tilde{u}_H\in \tilde{U}_H$ we have that
\begin{align*}
\sup_{\tilde{v}_H\in \tilde{V}_H\backslash\{0\}} \frac{\lvert a(\tilde{u}_H,\tilde{v}_H)\rvert}{\|\tilde{v}_H\|_V} \geq \frac{\lvert a(\tilde{u}_H,(1-\calC^{\ast})\tilde{u}_H)\rvert}{\|(1-\calC^{\ast})\tilde{u}_H\|_V} \geq \frac{\lvert a(\tilde{u}_H,\tilde{u}_H)\rvert}{\alpha^{-1}C_a \|\tilde{u}_H\|_V}  \geq \alpha^2 C_a^{-1} \|\tilde{u}_H\|_V,
\end{align*}
where we have used \cref{Lmm: Prop UHtilde}(ii), the fact that $\|1-\calC^{\ast}\| = \|\calC^{\ast}\|$ (see \cite{Szy06}), and that $\|\calC^{\ast}\| \leq \alpha^{-1} C_a$ by \cref{Rk: C well def}. Similarly, for any $\tilde{v}_H\in \tilde{V}_H$ we have that
\begin{align*}
\sup_{\tilde{u}_H\in \tilde{U}_H\backslash\{0\}} \frac{\lvert a(\tilde{u}_H,\tilde{v}_H)\rvert}{\|\tilde{u}_H\|_V} \geq \frac{\lvert a((1-\calC)\tilde{v}_H,\tilde{v}_H)\rvert}{\|(1-\calC)\tilde{v}_H\|_V} \geq \frac{\lvert a(\tilde{v}_H,\tilde{v}_H)\rvert}{\alpha^{-1}C_a \|\tilde{v}_H\|_V}  \geq \alpha^2 C_a^{-1} \|\tilde{v}_H\|_V.
\end{align*}
By the Babuška--Lax--Milgram theorem, there exists a unique solution $\tilde{u}_H\in \tilde{U}_H$ to the ideal discrete problem \cref{var problem discrete} and we obtain (i). It only remains to show (ii) and (iii).

(ii) We show that $u-\tilde{u}_H = \calC u \in W$. Observing that we have the Galerkin orthogonality (recall $\tilde{V}_H\subset V$) 
\begin{align*}
a(u-\tilde{u}_H,\tilde{v}_H) = 0\qquad \forall \tilde{v}_H\in \tilde{V}_H,
\end{align*}
we find that $u-\tilde{u}_H\in W$ by \cref{Lmm: Prop UHtilde}(ii) and \cref{a ort dec}. Finally, as $u-\tilde{u}_H\in W$, we have that $u-\tilde{u}_H = \calC(u-\tilde{u}_H) = \calC u - \calC \tilde{u}_H =  \calC u$. Here, we have used that $\calC \tilde{u}_H = 0$ by the definition of $\tilde{U}_H$ from \cref{UHtilde} and the properties of $\calC$ from \cref{Rk: C well def}. 

(iii) First, we note that by \cref{Rk: C well def,Rk: H2 bound on u} we have the bound
\begin{align*}
\|\calC u\|_{V} \leq \alpha^{-1} C_a\|u\|_V \leq  \alpha^{-2}\mu\, C_a  \|f\|_{L^2(\Omega)}.
\end{align*}
In view of the fact that $u-\tilde{u}_H = \calC u\in W$ (see (ii)) and using the bound \cref{bd for w in W}, we deduce that
\begin{align*}
\|u-\tilde{u}_H\|_{L^2(\Omega)} + H\|\nabla (u-\tilde{u}_H)\|_{L^2(\Omega)}  \lesssim H^2 \|\calC u\|_{V} \lesssim H^2\|f\|_{L^2(\Omega)},
\end{align*}
which concludes the proof.  
\end{proof}

\subsection{Construction of a coarse-scale space}

\subsubsection{Construction of a local basis}\label{Sec: constr locbas}

We are going to construct functions $u_1,\dots,u_N\in V$ with local support that satisfy $\langle q_i,u_j\rangle = \delta_{ij}$ for any $i,j\in \{1,\dots,N\}$.

To this end, we introduce the Hsieh--Clough--Tocher (HCT) finite element space
\begin{align*}
V_H^{\mathrm{HCT}}\coloneqq \{v\in V: \left. v\right\rvert_T \in \calP_3(\calK_H(T)) \;\forall\, T\in \calT_H\},
\end{align*}
where $\calK_H(T)$ denotes the triangulation of the triangle $T$ into three sub-triangles with shared vertex $\mathrm{mid}(T)$, and we make use of the HCT enrichment operator $E_H:V_H^{\mathrm{Mor}} \rightarrow V_H^{\mathrm{HCT}}$ defined in \cite[Prop. 2.5]{Gal15}. We then define the operator
\begin{align*}
\tilde{E}_H:V_H^{\mathrm{Mor}} \rightarrow V,\quad v \mapsto \tilde{E}_H v\coloneqq  E_H v + \sum_{i=1}^{N_{1}} \left[ \meanint_{F_i} \nabla(v-E_H v)\cdot  \nu \right]\zeta_{F_i},
\end{align*}
where $\zeta_{F_i}\in V$ is the function from \cite[Proof of Prop. 2.6]{Gal15} which satisfies
\begin{align*}
\meanint_{F_i} \nabla \zeta_{F_i}\cdot \nu = 1,\qquad \zeta_{F_i}(z) = 0\quad\forall z\in \calN_H^{\mathrm{int}},
\end{align*}
and $\supp(\zeta_{F_i}) \subset \bar{\omega}_{F_i}$, where $\bar{\omega}_{F_i}$ denotes the closure of the union of the two elements that share the edge $F_i$. For any $v_H^{\mathrm{Mor}}\in V_H^{\mathrm{Mor}}$, we have that
\begin{align}\label{EHtilde preserves qi}
\begin{split}
\meanint_{F} \nabla (\tilde{E}_H v_H^{\mathrm{Mor}})\cdot \nu &= \meanint_{F} \nabla v_H^{\mathrm{Mor}}\cdot \nu\qquad \forall F\in \calF_H, \\ (\tilde{E}_H v_H^{\mathrm{Mor}})(z) &= v_H^{\mathrm{Mor}}(z) \hspace{1.7cm} \forall z\in \calN_H^{\mathrm{int}},
\end{split}
\end{align}
i.e., $\tilde{E}_H$ preserves the quantities of interest $q_1,\dots,q_N$. Further, we have the bound
\begin{align*}
&\|h_{\calT_H}^{-2} (v_{H}^{\mathrm{Mor}}- E_H v_{H}^{\mathrm{Mor}})\|_{L^2(\Omega)} + \|h_{\calT_H}^{-1} \nabla_{\mathrm{NC}}(v_{H}^{\mathrm{Mor}}- E_H v_{H}^{\mathrm{Mor}})\|_{L^2(\Omega)} + \|D^2_{\mathrm{NC}}(v_{H}^{\mathrm{Mor}}-E_H v_{H}^{\mathrm{Mor}})\|_{L^2(\Omega)} \\&\lesssim \min_{v\in V}\|D^2_{\mathrm{NC}}(v_H^{\mathrm{Mor}}-v)\|_{L^2(\Omega)}\qquad \forall v_{H}^{\mathrm{Mor}}\in V_{H}^{\mathrm{Mor}},
\end{align*}
where the subscript $\mathrm{NC}$ indicates the piecewise action of a differential operator with respect to the triangulation $\calT_H$, and we have that
\begin{align}\label{EHtildebd}
\begin{split}
&\|h_{\calT_H}^{-2} (v_{H}^{\mathrm{Mor}}- \tilde{E}_H v_{H}^{\mathrm{Mor}})\|_{L^2(\Omega)} + \|h_{\calT_H}^{-1} \nabla_{\mathrm{NC}}(v_{H}^{\mathrm{Mor}}- \tilde{E}_H v_{H}^{\mathrm{Mor}})\|_{L^2(\Omega)} + \|D^2_{\mathrm{NC}}(v_{H}^{\mathrm{Mor}}- \tilde{E}_H v_{H}^{\mathrm{Mor}})\|_{L^2(\Omega)} \\&\lesssim \min_{v\in V}\|D^2_{\mathrm{NC}}(v_H^{\mathrm{Mor}}-v)\|_{L^2(\Omega)}\qquad \forall v_{H}^{\mathrm{Mor}}\in V_{H}^{\mathrm{Mor}}.
\end{split}
\end{align}
The proofs of \cite[Prop. 2.5--2.6]{Gal15} furthermore show the quasi-local bound
\begin{align}\label{EHtildebd_loc}
\begin{split}
&\|h_{\calT_H}^{-2} (v_{H}^{\mathrm{Mor}}- \tilde{E}_H v_{H}^{\mathrm{Mor}})\|_{L^2(T)} 
+ \|h_{\calT_H}^{-1} \nabla_{\mathrm{NC}}(v_{H}^{\mathrm{Mor}}- \tilde{E}_H v_{H}^{\mathrm{Mor}})\|_{L^2(T)} 
+ \|D^2_{\mathrm{NC}}(v_{H}^{\mathrm{Mor}}- \tilde{E}_H v_{H}^{\mathrm{Mor}})\|_{L^2(T)} 
\\
&\lesssim \min_{v\in V}\|D^2_{\mathrm{NC}}(v_H^{\mathrm{Mor}}-v)\|_{L^2(N^1(T))}\qquad \forall v_{H}^{\mathrm{Mor}}\in V_{H}^{\mathrm{Mor}}
\end{split}
\end{align}
for any $T\in\mathcal T_H$. We define the functions
\begin{align}\label{defn ui}
u_i \coloneqq  \tilde{E}_H\phi_i \in V,\qquad i\in \{1,\dots,N\},
\end{align}
where $\phi_1,\dots,\phi_N\in V_{H}^{\mathrm{Mor}}$ are the Morley basis functions from \cref{Sec: Conn to Mor}, and we set
\begin{align*}
U_H\coloneqq \mathrm{span}(u_1,\dots,u_N)\subset V.
\end{align*}
By \cref{EHtilde preserves qi} and the definition of the Morley basis functions there holds 
\begin{align}\label{uj main prop}
\langle q_i,u_j\rangle = \delta_{ij}\qquad \forall i,j\in \{1,\dots,N\}
\end{align}
and we have that $\Omega_i\coloneqq \supp(u_i)\subset N^1(\supp(\phi_i))$ for any $i\in\{1,\dots,N\}$. Further, we have the following result.
\begin{lmm}[Stability of basis representation]\label{Lmm: Stab of basis rep}
For any $u_H = \sum_{i=1}^N c_i u_i\in U_H$ with $c_i=\langle q_i,u_H\rangle$ for $i\in\{1,\dots,N\}$ we have that
\begin{align*}
\sum_{i=1}^N c_i^2 \|u_i\|_V^2 \lesssim H^{-4} \|u_H\|_V^2.
\end{align*}
\end{lmm}

\begin{proof}
Let $u_H = \sum_{i=1}^N c_i u_i \in U_H$ with $c_i=\langle q_i,u_H\rangle$ for $i\in\{1,\dots,N\}$. Then, by the definition \cref{defn ui} of $u_i$, the bound \cref{EHtildebd} for $\tilde{E}_H$, and inverse estimates for Morley functions, we have that 
\begin{align*}
\sum_{i=1}^N c_i^2 \|u_i\|_V^2 = \sum_{i=1}^N c_i^2 \|\tilde{E}_H \phi_i\|_{H^2(\Omega)}^2 \lesssim \sum_{i=1}^N c_i^2 \|\phi_i\|_{H^2(\Omega;\calT_H)}^2 \lesssim H^{-4}\sum_{T\in \calT_H}\sum_{\substack{i\in\{1,\dots,N\}\\ \supp(\phi_i)\cap T\neq \emptyset}}  \langle q_i, u_H\rangle^2 \|\phi_i\|_{L^2(T)}^2,
\end{align*}
where we have used the notation $H^2(\Omega;\calT_H):=\{\phi\in L^2(\Omega): \left.\phi\right\rvert_T \in H^2(T)\;\forall T\in \calT_H\}$ to denote the broken $H^2$-space, and $\|\cdot\|_{H^2(\Omega;\calT_H)}:=\sqrt{\sum_{T\in \calT_H} \|\cdot\|_{H^2(T)}^2}$ to denote the broken $H^2$-norm. We deduce that
\begin{align*}
\sum_{i=1}^N c_i^2 \|u_i\|_V^2&\lesssim H^{-4}\sum_{T\in \calT_H}\left\|\sum_{\substack{i\in\{1,\dots,N\}\\ \supp(\phi_i)\cap T\neq \emptyset}}  \langle q_i, u_H\rangle \phi_i\right\|_{L^2(T)}^2 \lesssim H^{-4}\sum_{T\in \calT_H}\left\|\Pi^{\mathrm{Mor}}u_H\right\|_{L^2(T)}^2 \lesssim H^{-4} \|u_H\|_{V}^2.
\end{align*}
In the final step, we have used that $\Pi^{\mathrm{Mor}}u_H= (\Pi^{\mathrm{Mor}}-1)u_H + u_H$ and a Morely interpolation estimate; see \cite{WX13}.
\end{proof}

\subsubsection{Projector onto $U_H$}

We introduce the projector
\begin{align*}
P_H: V\rightarrow U_H,\qquad v\mapsto P_H v\coloneqq  \sum_{i=1}^N \langle q_i,v\rangle u_i.
\end{align*}
\begin{rmrk}\label{Rk: PH char}
We can equivalently characterize $P_H$ as follows:
\begin{itemize}
\item[(i)] By \cref{defn ui,PiMo}, we have that
\begin{align*}
P_H v = \sum_{i=1}^N \langle q_i,v\rangle \tilde{E}_H\phi_i = \tilde{E}_H\left( \sum_{i=1}^N \langle q_i,v\rangle \phi_i \right) = \tilde{E}_H(\Pi^{\mathrm{Mor}} v)\qquad \forall v\in V,
\end{align*}  
that is, $P_H = \tilde{E}_H\circ \Pi^{\mathrm{Mor}}$.
\item[(ii)] In view of (i) and introducing $I_H\coloneqq  E_H\circ \Pi^{\mathrm{Mor}}$, we have that
\begin{align*}
P_H v = I_H v + \sum_{i=1}^{N_{1}} \langle q_i,v-I_H v\rangle \zeta_{F_i}\qquad \forall v\in V.
\end{align*}
\end{itemize}
\end{rmrk}

We list some stability properties of the projector $P_H$ below.

\begin{lmm}[Stability of $P_H$]\label{Lmm: Stab PH}
There exist constants $C_{P_H},C_{P_H}^{\mathrm{loc}}>0$ independent of $H$ such that we have the stability bound
\begin{align*}
\|P_H v\|_{V}\leq C_{P_H}\|v\|_V\qquad\forall v\in V
\end{align*}
and the local stability bound
\begin{align*}
\|P_H v\|_{V,S}\leq C_{P_H}^{\mathrm{loc}}\|v\|_{V,N^1(S)}\qquad\forall v\in V
\end{align*}
for any element patch $S$.
\end{lmm}

\begin{proof}
The global stability bound follows from the fact that $P_H = \tilde{E}_H\circ \Pi^{\mathrm{Mor}}$ and the estimate \cref{EHtildebd}. The local stability bound follows from the decomposition $
P_H=(\tilde{E}_H-1)\circ\Pi^{\mathrm{Mor}}+\Pi^{\mathrm{Mor}}$, the triangle inequality,
and the local bound \cref{EHtildebd_loc}.
\end{proof}

\begin{rmrk}[Properties of $P_H$]
We make the following observations:
\begin{itemize}
\item[(i)] For any $u_H\in U_H$ we have that $P_H (1-\calC)u_H = P_H u_H = u_H$ and $P_H (1-\calC^{\ast})u_H = u_H$.
\item[(ii)] There holds $\mathrm{ker}(P_H) = W$. 
\item[(iii)] For any $v\in V$ we have that $(1-P_H)v\in W$ and hence, there holds $(1-\calC)v = (1-\calC)P_H v$ and $(1-\calC^{\ast})v = (1-\calC^{\ast})P_H v$.
\end{itemize}

\end{rmrk}

\subsubsection{Connection of $\tilde{U}_H$ and $\tilde{V}_H$ to the space $U_H$}

First, let us note that in view of \cref{uj main prop} we have that $\dim(U_H)=N$ and $U_H\cap W = \{0\}$. Recalling that $W = \mathrm{ker}(1-\calC) = \mathrm{ker}(1-\calC^{\ast})$, we see that $\dim((1-\calC)U_H)=\dim((1-\calC^{\ast})U_H)=N$, and we deduce that
\begin{align*}
(1-\calC) U_H = \tilde{U}_H,\qquad (1-\calC^{\ast}) U_H = \tilde{V}_H.
\end{align*}
Note that $\{(1-\calC)u_1,\dots,(1-\calC)u_N\}$ is a basis of $\tilde{U}_H$ and that $\{(1-\calC^{\ast})u_1,\dots,(1-\calC^{\ast})u_N\}$ is a basis of $\tilde{V}_H$. It can be checked that the function $(1-\calC)u_i$ is independent of the particular choice of $u_i$ as indicated below.
\begin{rmrk}\label{rem:saddlepoint}
Using the arguments presented in \cite[Section 3.4]{AHP21}, it can be seen that for any $i\in\{1,\dots,N\}$ there exists a unique pair $(\tilde{u}_{i},\tilde{\lambda})\in V\times \R^N$ such that
\begin{align*}
\begin{cases}
a(\tilde{u}_{i},v) + \sum_{j=1}^N \tilde{\lambda}_{j} \langle q_{j},v\rangle &= 0\quad\;\, \forall v\in V,\\
\hfill \langle q_{j},\tilde{u}_{i}\rangle &= \delta_{ij}\quad\forall j\in\{1,\dots,N\}.
\end{cases}
\end{align*}
Further, there holds $\tilde{u}_{i} = (1-\calC)u_{i}$ and $\tilde{\lambda}_{j} = -a((1-\calC)u_{i},(1-\calC^{\ast})u_{j})$ for all $i,j\in\{1,\dots,N\}$.
\end{rmrk}

\subsection{Construction of localized correctors}

\subsubsection{Exponential decay of correctors}

The following lemma sets the foundation for the construction of a practical/localized numerical homogenization scheme. 

\begin{lmm}[Exponential decay of correctors]\label{Lmm: exp decay}
There exists a constant $\beta>0$ such that for any $v\in V$ and any $\ell\in \N_0$ we have
\begin{align*}
\|\calC v\|_{V,\Omega\backslash N^{\ell}(\Omega_v)}\lesssim \mathrm{e}^{-\frac{1}{25}\lvert \log(\beta)\rvert \ell}\,\|\calC v\|_V,
\end{align*}
where $\Omega_v\coloneqq \bigcup \{T\in \calT_H: T\cap \mathrm{supp}(v)\neq \emptyset\}$.
\end{lmm}

\begin{proof}
First, let us note that $\mathrm{supp}(P_H v)\subset N^1(S)$ for any $v\in V$ with $\mathrm{supp}(v)\subset S$, where $S$ is an element-patch in $\calT_H$. Let $v\in V$ and let $\ell\in \N$ with $\ell\geq 5$. 

Let $\eta\in W^{2,\infty}(\Omega)$ be a cut-off function with
\begin{align*}
0\leq \eta\leq 1,\quad \left.\eta\right\rvert_{N^{\ell-1}(\Omega_v)}\equiv 0,\quad \left.\eta\right\rvert_{\Omega\backslash N^{\ell}(\Omega_v)}\equiv 1,\quad \|\nabla \eta\|_{L^{\infty}(\Omega)} + H \|D^2 \eta\|_{L^{\infty}(\Omega)}\lesssim H^{-1},
\end{align*}
and let $\tilde{\eta}\coloneqq 1-\eta\in W^{2,\infty}(\Omega)$.
We introduce 
\begin{align}\label{Pf decay 0}
w\coloneqq (1-P_H)[\eta\, \calC v],\qquad \tilde{w}\coloneqq (1-P_H)[\tilde{\eta}\, \calC v]
\end{align}
and note that $w,\tilde{w}\in W$, there holds $\mathrm{supp}(w)\subset \Omega\backslash N^{\ell-2}(\Omega_v)$, $\mathrm{supp}(\tilde{w})\subset N^{\ell+1}(\Omega_v)$, and we have that
\begin{align}\label{Pf decay 1}
\|\calC v\|_{V,\Omega\backslash N^{\ell+1}(\Omega_v)}^2 = \|(1-P_H)[\calC v]\|_{V,\Omega\backslash N^{\ell+1}(\Omega_v)}^2 = \|w+\tilde{w}\|_{V,\Omega\backslash N^{\ell+1}(\Omega_v)}^2 \leq  \|w\|_{V}^2 \leq  \alpha^{-1} a(w,w),
\end{align}
where we have successively used that $P_H[\calC v] = 0$ as $\mathrm{ker}(P_H) = W$, the definition \cref{Pf decay 0} of the functions $w$ and $\tilde{w}$, the fact that $\supp(\tilde{w})\subset N^{\ell+1}(\Omega_v)$, and coercivity of $a$ from \cref{Lmm: prop of a F}(ii). Next, we observe that 
\begin{align}\label{Pf decay 2}
a(w,w) + a(\tilde{w},w)=  a((1-P_H)[\calC v],w)  = a(\calC v,w) =  a(v,w)  =0,
\end{align}  
where we have used bilinearity of $a$, the fact that $P_H[\calC v] = 0$, the definition of $\calC$, and the observation that in view of \cref{Lmm: prop of a F}(i) there holds $a(v,w) = 0$ as $\supp(v)\subset \Omega_v$ and $\supp(w)\subset  \Omega\backslash N^{\ell-2}(\Omega_v)$. Combining \cref{Pf decay 1,Pf decay 2}, and using \cref{Lmm: prop of a F}(i), we find that
\begin{align}\label{Pf decay 2.5}
\|\calC v\|_{V,\Omega\backslash N^{\ell+1}(\Omega_v)}^2 \leq \alpha^{-1} a(w,w) = -\alpha^{-1} a(\tilde{w},w) \leq \alpha^{-1} C_a \|\tilde{w}\|_{V,R} \|w\|_{V,R},
\end{align}
where $R\coloneqq  \mathrm{supp}(\tilde{w})\cap \mathrm{supp}(w) = N^{\ell+1}(\Omega_v)\backslash N^{\ell-2}(\Omega_v)$. We proceed by noting that by \cref{Lmm: Stab PH} we have that
\begin{align}\label{Pf of decay 3}
\|w\|_{V,R} \lesssim \|\eta\,\calC v\|_{V,N^1(R)} \lesssim \|\calC v\|_{V,N^1(R)}.
\end{align} 
Here, the final bound in \cref{Pf of decay 3} follows from the fact that for any $T\in N^1(R)$ there holds
\begin{align*}
\|\eta\,\calC v\|_{V,T} &\lesssim \| \eta\|_{L^{\infty}(\Omega)}\|\calC v\|_{L^2(T)} + \|\nabla \eta\|_{L^{\infty}(\Omega)}\|\calC v\|_{L^2(T)} + \| \eta\|_{L^{\infty}(\Omega)}\|\nabla [\calC v]\|_{L^2(T)}\\ &\qquad+\|D^2 \eta\|_{L^{\infty}(\Omega)} \|\calC v\|_{L^2(T)}  + \|\nabla \eta\|_{L^{\infty}(\Omega)}\|\nabla [\calC v]\|_{L^2(T)}+\| \eta\|_{L^{\infty}(\Omega)} \|D^2 [\calC v]\|_{L^2(T)}\\ &\lesssim \|\calC v\|_{V,T},
\end{align*}
where we have used the properties of the cut-off function $\eta$ and the bound \cref{bd for w in W loc} for the function $\calC v\in W$. Similarly to \cref{Pf of decay 3}, we find that 
\begin{align}\label{Pf of decay 4}
\|\tilde{w}\|_{V,R}\lesssim \|\tilde{\eta}\,\calC v\|_{V,N^1(R)} \lesssim \|\calC v\|_{V,N^1(R)}.
\end{align}
Combining \cref{Pf of decay 3}--\cref{Pf of decay 4} with \cref{Pf decay 2.5}, we obtain that there exists a constant $C>0$ such that
\begin{align*}
\|\calC v\|_{V,\Omega\backslash N^{\ell+2}(\Omega_v)}^2\leq \|\calC v\|_{V,\Omega\backslash N^{\ell+1}(\Omega_v)}^2 \leq C^2 \|\calC v\|_{V,N^1(R)}^2= C^2 (\|\calC v\|_{V,\Omega\backslash N^{\ell-3}(\Omega_v)}^2 - \|\calC v\|_{V,\Omega\backslash N^{\ell+2}(\Omega_v)}^2)
\end{align*}
and hence, setting $\beta \coloneqq  \frac{C}{\sqrt{1+C^2}}\in (0,1)$, we have that
\begin{align*}
\|\calC v\|_{V,\Omega\backslash N^{\ell+2}(\Omega_v)}\leq \beta \|\calC v\|_{V,\Omega\backslash N^{\ell-3}(\Omega_v)}.
\end{align*}
Setting $k\coloneqq  \floor*{\frac{\ell}{5}}$ and recalling that $\ell\geq 5$, a repeated application of this bound yields
\begin{align*}
\|\calC v\|_{V,\Omega\backslash N^{\ell}(\Omega_v)}\leq \beta^k\|\calC v\|_{V} = \mathrm{e}^{-\frac{k}{\ell}\lvert \log(\beta)\rvert \ell}\|\calC v\|_{V} \leq \mathrm{e}^{-\frac{\ell-4}{5l}\lvert \log(\beta)\rvert \ell}\|\calC v\|_{V} \leq \mathrm{e}^{-\frac{1}{25}\lvert \log(\beta)\rvert \ell}\|\calC v\|_{V},
\end{align*}
proving the claim for the case $\ell\geq 5$. Finally, note that for $\ell\in \N_0$ with $\ell<5$ we have
\begin{align*}
\|\calC v\|_{V,\Omega\backslash N^{\ell}(\Omega_v)}\leq \|\calC v\|_V \leq \mathrm{e}^{\frac{1}{5}\lvert \log(\beta)\rvert} \mathrm{e}^{-\frac{1}{25}\lvert \log(\beta)\rvert \ell}\,\|\calC v\|_V,
\end{align*}
which concludes the proof.
\end{proof}

Using similar arguments, one obtains an analogous exponential decay result for the corrector $\calC^{\ast}$.

\subsubsection{Localized correctors}

Motivated by the fact that for any $u_H \in U_H$ we have that
\begin{align*}
\calC u_H = \sum_{i=1}^N \langle q_i,u_H\rangle\, \fhi_i,\qquad \calC^{\ast} u_H = \sum_{i=1}^N \langle q_i,u_H\rangle\, \psi_i,\qquad\text{where}\quad \fhi_i\coloneqq \calC u_i,\quad \psi_i\coloneqq \calC^{\ast} u_i,
\end{align*}
we define for $\ell\in \N_0$ the localized correctors 
\begin{align*}
\calC_\ell&:U_H\rightarrow W,\qquad u_H\mapsto \calC_\ell u_H\coloneqq \sum_{i=1}^N \langle q_i,u_H\rangle\, \fhi_i^\ell,\\
\calC_\ell^{\ast}&:U_H\rightarrow W,\qquad u_H\mapsto \calC_\ell^{\ast} u_H\coloneqq \sum_{i=1}^N \langle q_i,u_H\rangle\, \psi_i^\ell.
\end{align*}
Here, for $i\in\{1,\dots,N\}$, the functions $\fhi_i^\ell,\psi_i^\ell$ are defined as the unique $\fhi_i^\ell,\psi_i^\ell\in W(N^\ell(\Omega_i))$ that satisfy
\begin{align*}
a(\fhi_i^\ell,w) &= a(u_i,w)\qquad \forall w\in W(N^\ell(\Omega_i)), \\ a(w,\psi_i^\ell) &= a(w,u_i)\qquad \forall w\in W(N^\ell(\Omega_i)),
\end{align*}
where we write $W(N^\ell(\Omega_i))\coloneqq \{w\in W: \supp(w)\subset N^\ell(\Omega_i)\}$. Note that $\calC_\ell$ and $\calC_\ell^{\ast}$ are well-defined by the properties of $a$ from \cref{Lmm: prop of a F}.

\subsubsection{Localization error}

We can quantify the error committed in approximating the true correctors $\calC,\calC^{\ast}$ by their localized counterparts $\calC_\ell,\calC^{\ast}_\ell$. 

\begin{thrm}[Localization error for corrector]
There exists a constant $s>0$ such that for any $u_H\in U_H$ and $\ell\in \N_0$ there holds
\begin{align*}
\|(\calC - \calC_\ell) u_H\|_V \lesssim H^{-2} \sqrt{N}\, \mathrm{e}^{-s\ell}\|u_H\|_V.
\end{align*}
\end{thrm}

\begin{proof}
First, suppose $\ell\geq 4$. Note that the functions $\fhi_i = \calC u_i$ and $\fhi_i^\ell$ are uniquely characterized as solutions to the following problems:
\begin{align*}
\fhi_i\in W,\qquad a(\fhi_i,w) &= a(u_i,w)\quad\forall w\in W,\\
\fhi_i^\ell\in W(N^\ell(\Omega_i)),\qquad a(\fhi_i^\ell,w) &= a(u_i,w)\quad\forall w\in W(N^\ell(\Omega_i)).
\end{align*}
Therefore, as $W(N^\ell(\Omega_i))\subset W$, we can use the properties of $a$ from \cref{Lmm: prop of a F} and Galerkin-orthogonality to find that
\begin{align}\label{Pf of loc err 0}
\|\fhi_i - \fhi_i^\ell\|_V \leq \alpha^{-1} C_a \inf_{w\in W(N^\ell(\Omega_i))}\|\fhi_i-w\|_V.
\end{align}
Let $\eta\in W^{2,\infty}(\Omega)$ be a cut-off function with 
\begin{align*}
0\leq \eta\leq 1,\quad \left.\eta\right\rvert_{N^{\ell-2}(\Omega_i)}\equiv 1,\quad \left.\eta\right\rvert_{\Omega\backslash N^{\ell-1}(\Omega_i)}\equiv 0,\quad \|\nabla \eta\|_{L^{\infty}(\Omega)} + H \|D^2 \eta\|_{L^{\infty}(\Omega)}\lesssim H^{-1}.
\end{align*}
Then, setting $w^\ell\coloneqq (1-P_H)[\eta \fhi_i]\in W(N^\ell(\Omega_i))$, we have that
\begin{align*}
\|\fhi_i - \fhi_i^\ell\|_V &\leq \alpha^{-1} C_a \|\fhi_i-w^\ell\|_V= \alpha^{-1} C_a\|(1-P_H)[(1-\eta)\fhi_i]\|_{V,\Omega\backslash N^{\ell-3}(\Omega_i)} \lesssim \|\fhi_i\|_{V,\Omega\backslash N^{\ell-4}(\Omega_i)},
\end{align*}
where we have used \cref{Pf of loc err 0}, the fact that $\mathrm{ker}(P_H)=W$, and an argument analogous to \cref{Pf of decay 4} for the final bound. By the exponential decay property for $\calC$ from \cref{Lmm: exp decay}, we obtain that
\begin{align}\label{Pf of loc err 1}
\|\fhi_i - \fhi_i^\ell\|_V\lesssim \| \calC u_i\|_{V,\Omega\backslash N^{\ell-4}(\Omega_i)}\lesssim \mathrm{e}^{-s\ell}\| \calC u_i\|_{V}\lesssim \mathrm{e}^{-s\ell}\|  u_i\|_{V}
\end{align}
for some constant $s>0$, where we have used \cref{Rk: C well def} in the final step. Using the triangle inequality, the bound \cref{Pf of loc err 1}, and the Cauchy--Schwarz inequality, we find that
\begin{align*}
\|(\calC - \calC_\ell) u_H\|_V = \left\| \sum_{i=1}^N \langle q_i,u_H\rangle(\fhi_i- \fhi_i^\ell)\right\|_V \lesssim \mathrm{e}^{-s\ell} \sum_{i=1}^N \lvert \langle q_i,u_H\rangle\rvert \|  u_i\|_{V} \lesssim  \sqrt{N}\mathrm{e}^{-s\ell}\sqrt{\sum_{i=1}^N \lvert \langle q_i,u_H\rangle \rvert^2 \|  u_i\|_{V}^2}
\end{align*}
and hence, by \cref{Lmm: Stab of basis rep},
\begin{align*}
\|(\calC - \calC_\ell) u_H\|_V \lesssim H^{-2} \sqrt{N}\, \mathrm{e}^{-s\ell}\|u_H\|_V.
\end{align*}
Finally, in the case $\ell<4$, we have from \cref{Pf of loc err 0,Rk: C well def} that
\begin{align*}
\| \fhi_i - \fhi_i^\ell \|_V \leq \alpha^{-1}C_a \|\calC u_i\|_V \leq \alpha^{-2}C_a^2 \|u_i\|_V \leq \alpha^{-2}C_a^2\, \mathrm{e}^{4s}\,  \mathrm{e}^{-s\ell}\|u_i\|_V \lesssim \mathrm{e}^{-s\ell}\|  u_i\|_{V}, 
\end{align*} 
and we can conclude as before.
\end{proof}

Using similar arguments, one obtains an analogous result for the corrector $\calC^{\ast}$ and its localized counterpart $\calC^{\ast}_{\ell}$.

\subsection{Localized numerical homogenization scheme}

We are now in a position to state and analyze the practical numerical homogenization method.

\subsubsection{The localized numerical homogenization scheme}

We define the $N$-dimensional spaces
\begin{align*}
\tilde{U}_H^\ell \coloneqq  (1-\calC_\ell)U_H,\qquad \tilde{V}_H^\ell \coloneqq  (1-\calC_\ell^{\ast})U_H.
\end{align*}
Then, the numerical homogenization scheme reads as follows:
\begin{align}\label{numhom scheme}
\text{Find }\tilde{u}_H^\ell\in \tilde{U}_H^\ell\text{ such that }\quad a(\tilde{u}_H^\ell,\tilde{v}_H^\ell) = \langle F,\tilde{v}_H^\ell\rangle\quad \forall \tilde{v}_H^\ell\in \tilde{V}_H^\ell.
\end{align}

\subsubsection{Analysis of the localized numerical homogenization scheme}

The following theorem provides well-posedness of \cref{numhom scheme} as well as error bounds.

\begin{thrm}[Analysis of the localized numerical homogenization scheme]\label{Thm: Analysis LOD}
For $\ell\gtrsim \log(H^{-2}\sqrt{N})$ sufficiently large, there exists a unique solution $\tilde{u}_H^\ell\in \tilde{U}_H^\ell$ to \cref{numhom scheme}. Further, denoting the unique strong solution to \cref{u prob} by $u\in V$ and the unique solution to the ideal discrete problem \cref{var problem discrete} by $\tilde{u}_H\in\tilde{U}_H$, there exists a constant $s>0$ such that the following assertions hold true.
\begin{itemize}
\item[(i)] There holds \begin{align*}
\|P_H (\tilde{u}_H - \tilde{u}_H^\ell)\|_V \lesssim H^{-2} \sqrt{N} \mathrm{e}^{-s\ell}\|u\|_V.
\end{align*}
\item[(ii)] We have the bound
\begin{align*}
\lvert \langle q_i, u - \tilde{u}_H^\ell\rangle \rvert =\lvert \langle q_i, \tilde{u}_H - \tilde{u}_H^\ell\rangle \rvert  \lesssim H^{-2} \sqrt{N} \mathrm{e}^{-s\ell}\|u\|_V\qquad \forall i\in\{1,\dots,N\}
\end{align*}
for the error in the quantities of interest.
\item[(iii)] We have the error bound
\begin{align*}
\|u-\tilde{u}_H^\ell\|_{H^1(\Omega)} \lesssim \left(H  + H^{-2} \sqrt{N} \mathrm{e}^{-s\ell}\right)\|f\|_{L^2(\Omega)}.
\end{align*}
\end{itemize}
\end{thrm}

\begin{proof}
The well-posedness of \cref{numhom scheme} and the bounds from (i)--(ii) can be shown using identical arguments as in \cite{AHP21}. It remains to prove assertion (iii). To this end, we first use the triangle inequality, \cref{Thm: Analysis ideal method}, and (i) to find that
\begin{align}\label{Pf loc numhom 1}
\begin{split}
\|u-\tilde{u}_H^\ell\|_{H^1(\Omega)} &\leq \|u - \tilde{u}_H\|_{H^1(\Omega)} + \|P_H(\tilde{u}_H- \tilde{u}_H^\ell)\|_{H^1(\Omega)} +  \|(1-P_H)(\tilde{u}_H-\tilde{u}_H^\ell)\|_{H^1(\Omega)} \\ &\lesssim H \|f\|_{L^2(\Omega)} + H^{-2} \sqrt{N} \mathrm{e}^{-s\ell}\|f\|_{L^2(\Omega)} + \|(1-P_H)\tilde{e}_H^{\ell}\|_{H^1(\Omega)}
\end{split}
\end{align}
for some constant $s>0$, where $\tilde{e}_H^{\ell}\coloneqq \tilde{u}_H-\tilde{u}_H^\ell$. Next, using the triangle inequality, \cref{Rk: PH char}(i), the bound \cref{EHtildebd}, and a Morley interpolation estimate, we obtain that
\begin{align}\label{Pf loc numhom 2}
\begin{split}
\|(1-P_H)\tilde{e}_H^{\ell}\|_{H^1(\Omega)}&\leq \|(1-\tilde{E}_H)\Pi^{\mathrm{Mor}}\tilde{e}_H^{\ell}\|_{H^1(\Omega)} + \|(1-\Pi^{\mathrm{Mor}})\tilde{e}_H^{\ell}\|_{H^1(\Omega)} \\&\lesssim H\|D^2_{\mathrm{NC}}(\Pi^{\mathrm{Mor}}\tilde{e}_H^{\ell} - \tilde{e}_H^{\ell})\|_{L^2(\Omega)} + H \|\tilde{e}_H^{\ell}\|_{V} \\&\lesssim H\|\tilde{e}_H^{\ell}\|_{V}.
\end{split}
\end{align}
Finally, by the triangle inequality, \cref{Thm: Analysis ideal method}(ii), \cref{Rk: C well def} and quasi-optimality of the Petrov--Galerkin scheme \cref{numhom scheme,Rk: H2 bound on u}, we have that
\begin{align}\label{Pf loc numhom 3}
\|\tilde{e}_H^{\ell}\|_{V} \leq \|u-\tilde{u}_H\|_{V} + \|u-\tilde{u}_H^\ell\|_V = \|\calC u\|_{V} + \|u-\tilde{u}_H^\ell\|_V \lesssim \|u\|_V\lesssim \|f\|_{L^2(\Omega)}.
\end{align}
Combining \cref{Pf loc numhom 1}--\cref{Pf loc numhom 3} yields the desired bound.
\end{proof}

\section{Numerical experiments}\label{sec:Numerical Experiments}
In this section, we numerically investigate the proposed numerical homogenization scheme for nondivergence-form PDEs, which we abbreviate as LOD, due to its origin. We compare it to a conforming Birkhoff--Mansfield finite element method on the respective coarse mesh with mesh size $H$, denoted as FEM in the convergence plots. To simplify the presentation, $ H $ and $ h $ denote the minimal side lengths of the elements instead of their diameters in the remainder of this section.

\subsection{Conforming discretization}\label{sec:Fine-scale discretization by conforming Birkhoff--Mansfield elements}
The method presented in \cref{sec:Numerical homogenization scheme} is not yet discrete as it relies on the solution of the localized version of the saddle-point problem presented in \cref{rem:saddlepoint}, which is still in continuous form. For a finite element discretization we choose a finite-dimensional subspace $V_h \subset V$.
Here we choose $V_h$ to be the $H^2$-conforming reduced 
Birkhoff--Mansfield element space \cite{BM74,Cia78}.
Given a triangle $T$, we define the space $X(T)$ to be the sum of the
space of tricubic polynomials over $T$ (the polynomials that are
cubic when restricted to any line parallel to one of the triangle's edges)
and the three rational functions $\lambda_{1}^2\lambda_{2}/(1-\lambda_{3})$
(with cyclic permutation of the indices $1,2,3$).
The local shape function space $\mathrm{BM}(T)$ is then the
nine-dimensional subspace of $X(T)$
consisting of those functions whose normal derivative
on any of the three edges is affine.
The space $V_h$ is defined as the subspace of $H^2(\Omega)\cap H^1_0(\Omega)$
of functions that belong to $\mathrm{BM}(T)$ when restricted to any 
triangle $T$.
The nine local degrees of freedom are the point evaluation of the function
and of its gradient in the three vertices of any triangle.
Fore more details on the method and its variants, we refer to
\cite{BM74}.

\subsection{Implementation}

In all our experiments we consider the computational domain 
\begin{align*}
\Omega \coloneqq  (-1,1)^2\subset \R^2
\end{align*}
and use a mesh $\calT_h$ for the fine-scale discretization that resolves all small oscillations of the coefficients. We evaluate relative errors in the $L^2$-norm, the $H^1$-seminorm and the $H^2$-seminorm with respect to a reference solution originating from the fine mesh $\calT_h$ with $h=2^{-8}$. For the implementation, we used \texttt{Matlab} and extended the code provided in \cite{MalP20}.

For the demonstration of the multiscale method, we consider three choices of heterogeneous coefficients in combination with a right-hand side $f\in\{f^{(1)}, f^{(2)}, f^{(3)}\}$, where $f^{(2)}:=f^{(1)}$ and the functions $f^{(1)},f^{(3)}$ are given by
\begin{align*}
f^{(1)}(x)\coloneqq  (x_1 + \cos(3 \pi x_1)) x_2^3,\qquad f^{(3)}(x) \coloneqq  f^{(1)}(x) + 2\,\Theta(x_1)
\end{align*}
for $x = (x_1,x_2)\in \Omega$, where $\Theta(t):=\mathbbm{1}_{t>0}$ for $t\in \R$. Note that $f^{(3)}\in L^2(\Omega)\backslash H^1(\Omega)$.

\subsection{Periodic coefficient example}

We begin by considering a configuration with periodic coefficients. The coefficient $A$ is chosen as $A\coloneqq  A^{(1)}$, where $A^{(1)} \coloneqq  \left.\tilde{A}^{(1)}\left(\frac{\cdot}{\eps}\right)\right\rvert_{\Omega}$ with $\eps \coloneqq  2^{-6}$ and $\tilde{A}^{(1)}:\R^2\rightarrow\R^{2\times 2}_{\mathrm{sym}}$ is defined as
\begin{align*}
	\tilde{A}^{(1)}(x) \coloneqq  \begin{pmatrix}
		\frac{11}{4} + \frac{1}{4} \sin(\pi x_1)\cos(\pi x_2) & \sign(\sin(\pi x_1) \sin(\pi x_2)) \\
		\sign(\sin(\pi x_1) \sin(\pi x_2)) & \frac{7}{2} + \frac{1}{2} \cos^2(\pi x_1) \\
	\end{pmatrix}
\end{align*}
for $x = (x_1,x_2)\in\R^2$. We perform two numerical experiments with this periodic coefficient $A = A^{(1)}$. The right-hand side is chosen as $f:= f^{(1)}$. 

\subsubsection{Experiment 1: Vanishing lower-order terms}
\Cref{fig:conv_periodic} shows the corresponding errors for the case of vanishing lower-order terms ($\lvert b\rvert = c = 0$ a.e. in $\Omega$). Note that the Cordes condition \eqref{CordesC1} is satisfied by \cref{Rk: C1 in 2d}.

\subsubsection{Experiment 2: Non-vanishing lower-order terms}
For the case of non-vanishing lower order terms, we choose $b:= b^{(1)}$ and $c:= c^{(1)}$, where $b^{(1)} \coloneqq  \left.\tilde{b}^{(1)}\left(\frac{\cdot}{\eps}\right)\right\rvert_{\Omega}$, $c^{(1)} \coloneqq  \left.\tilde{c}^{(1)}\left(\frac{\cdot}{\eps}\right)\right\rvert_{\Omega}$ with $\eps =  2^{-6}$  and $\tilde{b}^{(1)}:\R^2\rightarrow \R^2$, $\tilde{c}^{(1)}:\R^2\rightarrow \R$ are given by
\begin{align*}
\tilde{b}^{(1)}(x) \coloneqq \begin{pmatrix}
\frac{3}{5} \sign(\sin(\pi x_1) \sin(\pi x_2))\\
\arcsin(\sin^2(\pi x_1)) - \frac{4}{5}
\end{pmatrix}, \qquad \tilde{c}^{(1)}(x)\coloneqq \frac{29}{10} + \frac{1}{10} \sign(\sin(\pi x_1) \sin(\pi x_2))
\end{align*}
for $x = (x_1,x_2)\in\R^2$. Note that the Cordes condition \eqref{CordesC2} is satisfied with $\lambda = 1$ since 
\begin{align*}
\frac{\lvert A^{(1)}\rvert^2 + \frac{1}{2}\lvert b^{(1)}\rvert^2 + \left(c^{(1)}\right)^2}{\left(\mathrm{tr}(A^{(1)}) + c^{(1)}\right)^2}\leq \frac{27+\frac{1}{2}+9}{(6+\frac{14}{5})^2} = \frac{1}{2 + \frac{222}{1825}}\quad\text{a.e. in }\Omega.
\end{align*}
The corresponding errors are depicted in \cref{fig:conv_periodic_lo}.

\begin{figure}
	\input{conv_BM_morley_BM_coeff_periodic_lambda_0_ref_9/conv_BM_morley_BM_coeff_periodic_coefflvl_7_rhs_C_ref_9_Hmin_1_Hmax_6_lmin_1_lmax_8_interpolation_nodal.tex}
	\caption{\small Relative errors for the periodic problem with vanishing lower-order terms ($A = A^{(1)}$, $b = 0$, $c = 0$, $f = f^{(1)}$).}
	\label{fig:conv_periodic}
\end{figure}

\begin{figure}
	\input{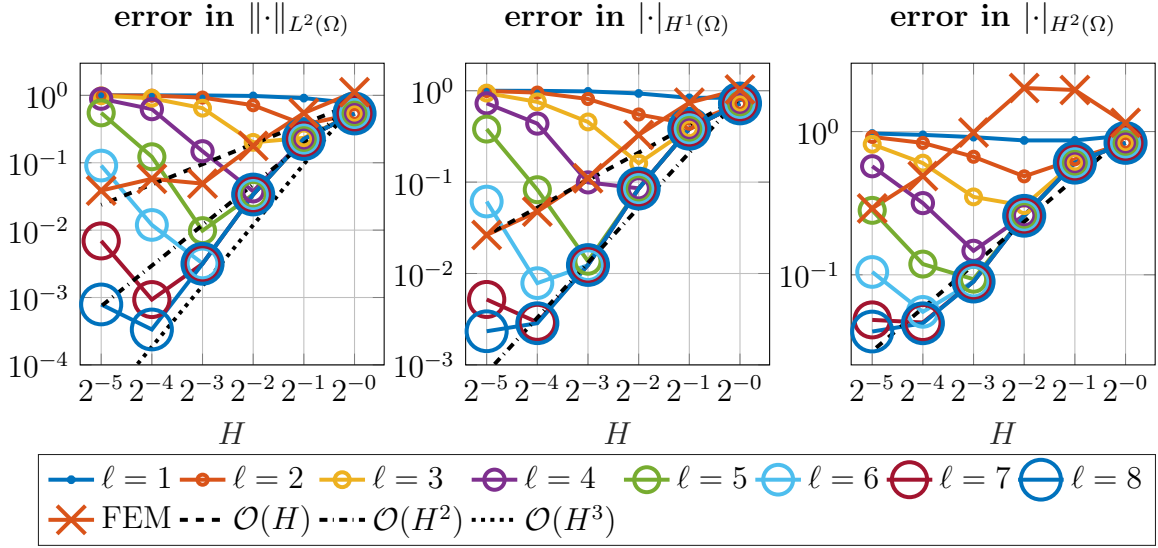}
	\caption{\small Relative errors for the periodic problem with non-vanishing lower-order terms ($A = A^{(1)}$, $b = b^{(1)}$, $c = c^{(1)}$, $f = f^{(1)}$) and $\lambda=1$.}
	\label{fig:conv_periodic_lo}
\end{figure}

\subsection{Crack coefficient example}\label{Subsec: Cr ex}
Next, we consider an example with $A:=A^{(2)}$, where 
\begin{align*}
	A^{(2)} \coloneqq  \begin{pmatrix}
		2 & a^{(2)}_{12} \\
		a^{(2)}_{12} & 2
	\end{pmatrix}
\end{align*}
and $a^{(2)}_{12}$ is the realization of a background random field taking piecewise constant values on $\calT_{2^{-6}}$, which are independent and identically distributed in the interval $[-1,-0.9] $, which is combined with a channel taking values close to $1$ (note $\lvert a^{(2)}_{12}\rvert \leq 1$); see \Cref{subfig:crack_A12}. We perform two numerical experiments with this ``crack coefficient" $A = A^{(2)}$. The right-hand side is chosen as $f:= f^{(2)}$.

\subsubsection{Experiment 1: Vanishing lower-order terms}
\cref{fig:conv_crack} shows the corresponding errors for the case of vanishing lower-order terms ($\lvert b\rvert = c = 0$ a.e. in $\Omega$). Note that the Cordes condition \eqref{CordesC1} is satisfied by \cref{Rk: C1 in 2d}. 

\subsubsection{Experiment 2: Non-vanishing lower-order terms}

For the case of non-vanishing lower order terms, we choose $b:= b^{(2)} := (b^{(2)}_1,b^{(2)}_2)$ and $c:=c^{(2)}$ that contain cracks at a different position than $a_{12}^{(2)}$. The function $b^{(2)}_1$ consists of a background random field taking values in the interval $[-0.1,0.1]$ with a crack that varies in the interval $[-1,1]$. For $b^{(2)}_2$, the random background is identical, whereas the crack varies in $[-0.6,0.6]$. Finally, $c^{(2)}$ consists of a random background taking values in $[3,3.1]$ with a crack varying in $[3,4]$. \Cref{fig:coeff_crack} depicts plots of these coefficients. Note that the Cordes condition \eqref{CordesC2} is satisfied with $\lambda = 2$ since 
\begin{align*}
\frac{\lvert A^{(2)}\rvert^2 + \frac{1}{4}\lvert b^{(2)}\rvert^2 + \frac{1}{4}\left(c^{(2)}\right)^2}{\left(\mathrm{tr}(A^{(2)}) + \frac{1}{2}c^{(2)}\right)^2}\leq \frac{10+\frac{17}{50}+4}{(4+\frac{3}{2})^2} = \frac{1}{2+\frac{157}{1434}}\quad\text{a.e. in }\Omega.
\end{align*}
The corresponding errors are depicted in \cref{fig:conv_crack_lo}.

\begin{figure}
\begin{subfigure}{0.32\textwidth}
		\includegraphics[width=\textwidth]{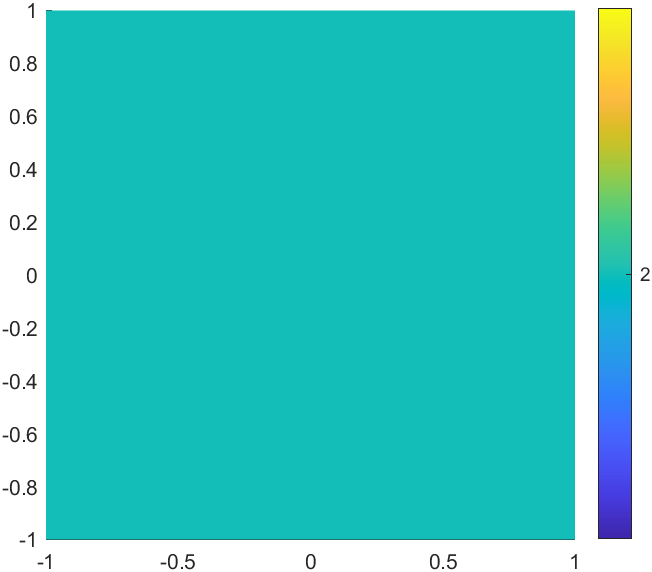}
		\subcaption{$a^{(2)}_{11}$}
		\label{subfig:crack_A11}
	\end{subfigure}
	\begin{subfigure}{0.32\textwidth}
		\includegraphics[width=\textwidth]{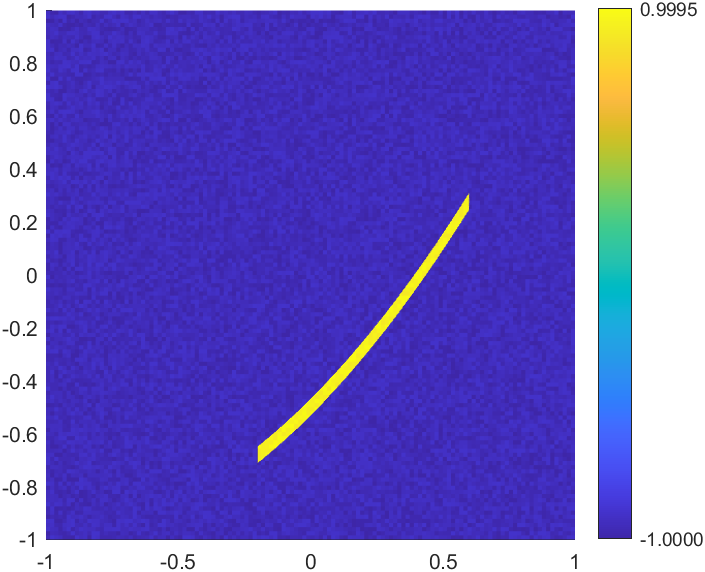}
		\subcaption{$a^{(2)}_{12}$}
		\label{subfig:crack_A12}
	\end{subfigure}
	\begin{subfigure}{0.32\textwidth}
		\includegraphics[width=\textwidth]{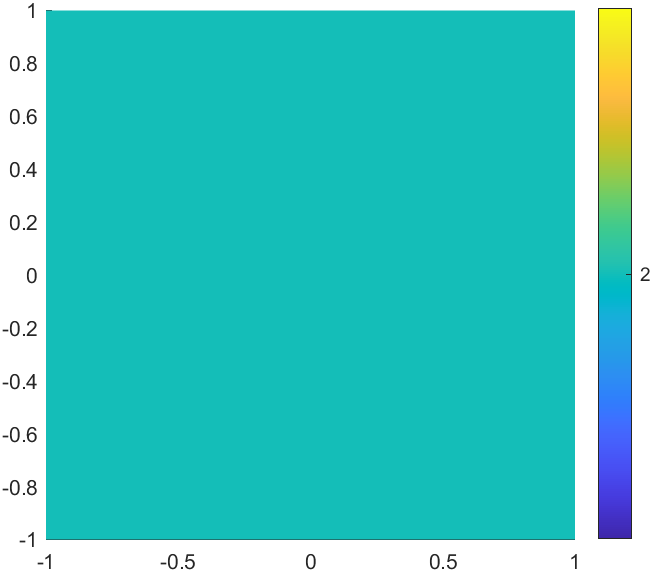}
		\subcaption{$a^{(2)}_{22}$}
		\label{subfig:crack_A22}
	\end{subfigure}
	\begin{subfigure}{0.32\textwidth}
		\includegraphics[width=\textwidth]{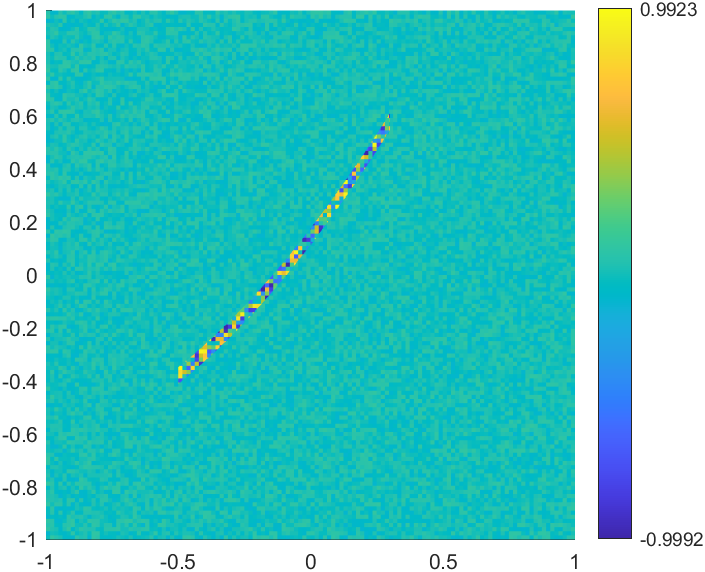}
		\subcaption{$b^{(2)}_{1}$}
		\label{subfig:crack_b1}
	\end{subfigure}
	\begin{subfigure}{0.32\textwidth}
		\includegraphics[width=\textwidth]{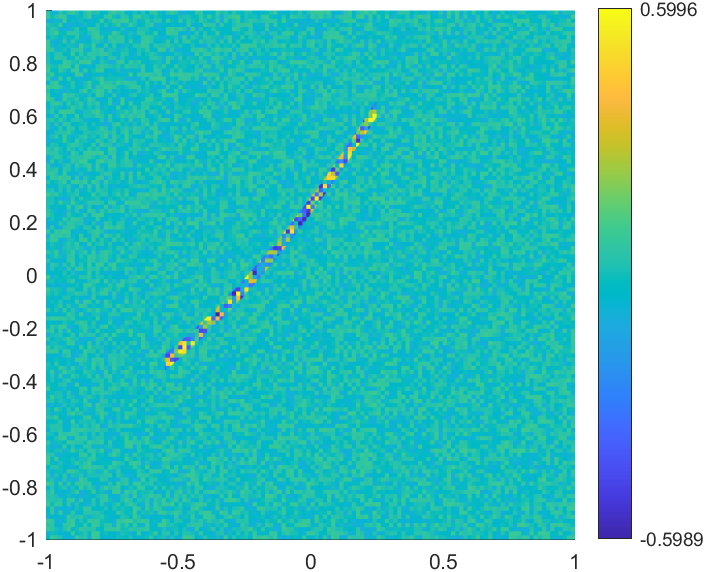}
		\subcaption{$b^{(2)}_{2}$}
		\label{subfig:crack_b2}
	\end{subfigure}
	\begin{subfigure}{0.32\textwidth}
		\includegraphics[width=\textwidth]{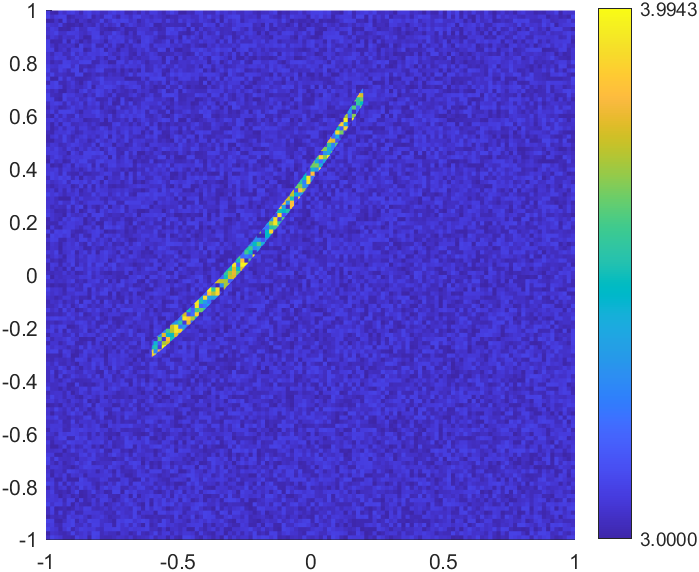}
		\subcaption{$c^{(2)}$}
		\label{subfig:crack_c}
	\end{subfigure}
	\caption{Illustration of the coefficients chosen in \cref{Subsec: Cr ex}.}
	\label{fig:coeff_crack}
\end{figure}

\begin{figure}
	\input{conv_BM_morley_BM_coeff_CrackI_lambda_0_ref_9/conv_BM_morley_BM_coeff_CrackI_coefflvl_7_rhs_C_ref_9_Hmin_1_Hmax_6_lmin_1_lmax_8_interpolation_nodal.tex}
	\caption{\small Relative errors for the crack problem with vanishing lower-order terms ($A = A^{(2)}$, $b = 0$, $c = 0$, $f = f^{(2)}$).}
	\label{fig:conv_crack}
\end{figure}

\begin{figure}
	\input{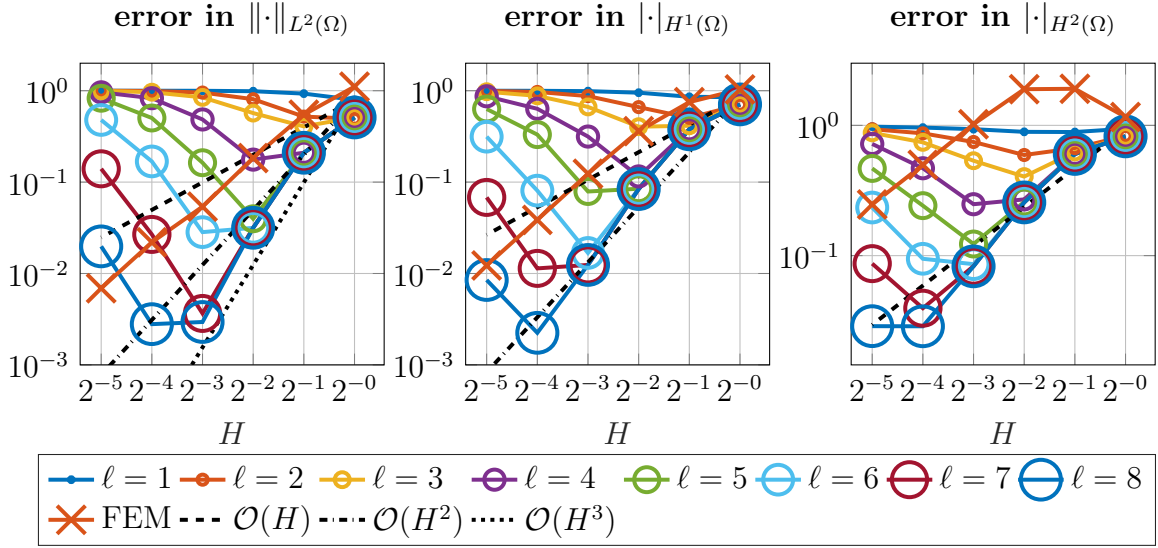}
	\caption{\small Relative errors for the crack problem with non-vanishing lower-order terms ($A = A^{(2)}$, $b = b^{(2)}$, $c = c^{(2)}$, $f = f^{(2)}$) and $\lambda=2$.}
	\label{fig:conv_crack_lo}
\end{figure}

\subsection{Combined example}\label{Subsec: combicoeff}
The final example combines various types of heterogeneities and the right-hand side is chosen as $f := f^{(3)}\in L^2(\Omega)\backslash H^1(\Omega)$. A plot of the chosen coefficients $A:=A^{(3)}$, $b:=b^{(3)}$, and $c:=c^{(3)}$ is given in \cref{fig:coeff_combi}. Note that the Cordes condition \eqref{CordesC2} is satisfied with $\lambda = 1$ since 
\begin{align*}
\frac{\lvert A^{(3)}\rvert^2 + \frac{1}{2}\lvert b^{(3)}\rvert^2 + \left(c^{(3)}\right)^2}{\left(\mathrm{tr}(A^{(3)}) + c^{(3)}\right)^2}\leq \frac{27+\frac{1}{2}+9}{(6+\frac{14}{5})^2}= \frac{1}{2 + \frac{222}{1825}} \quad\text{a.e. in }\Omega.
\end{align*}
The corresponding errors are depicted in \cref{fig:conv_combi_lo}.

\begin{figure}
	\begin{subfigure}{0.32\textwidth}
		\includegraphics[width=\textwidth]{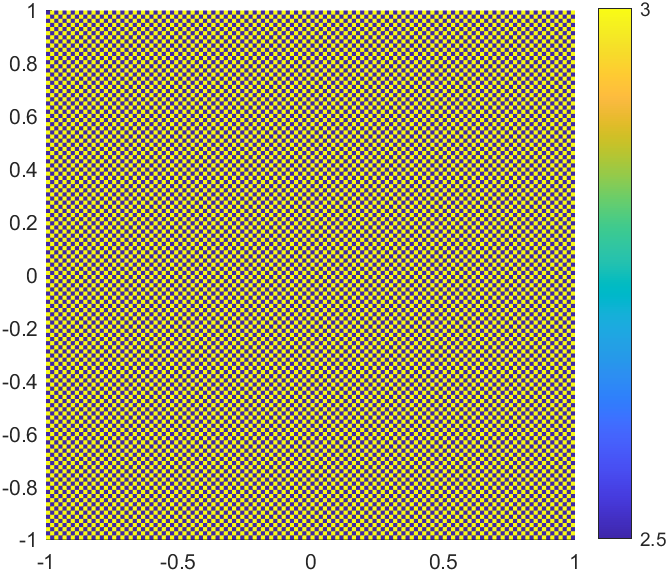}
		\subcaption{$a^{(3)}_{11}$}
		\label{subfig:combi_A11}
	\end{subfigure}
	\begin{subfigure}{0.32\textwidth}
		\includegraphics[width=\textwidth]{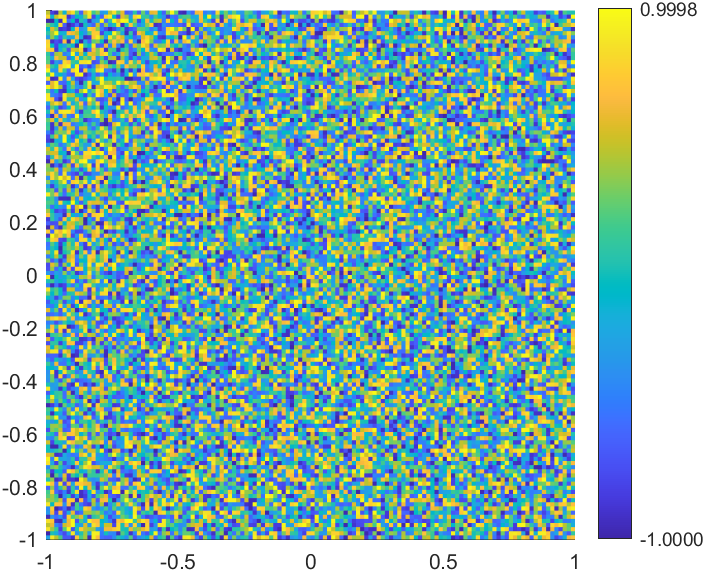}
		\subcaption{$a^{(3)}_{12}$}
		\label{subfig:combi_A12}
	\end{subfigure}
	\begin{subfigure}{0.32\textwidth}
		\includegraphics[width=\textwidth]{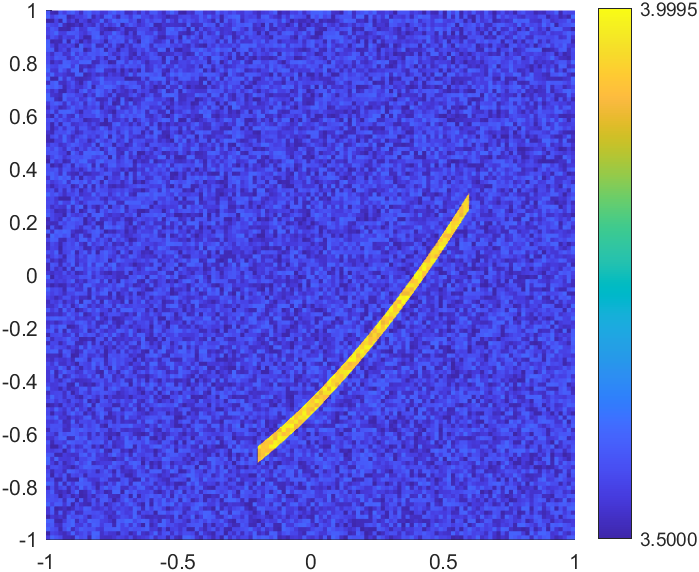}
		\subcaption{$a^{(3)}_{22}$}
		\label{subfig:combi_A22}
	\end{subfigure}
	\begin{subfigure}{0.32\textwidth}
		\includegraphics[width=\textwidth]{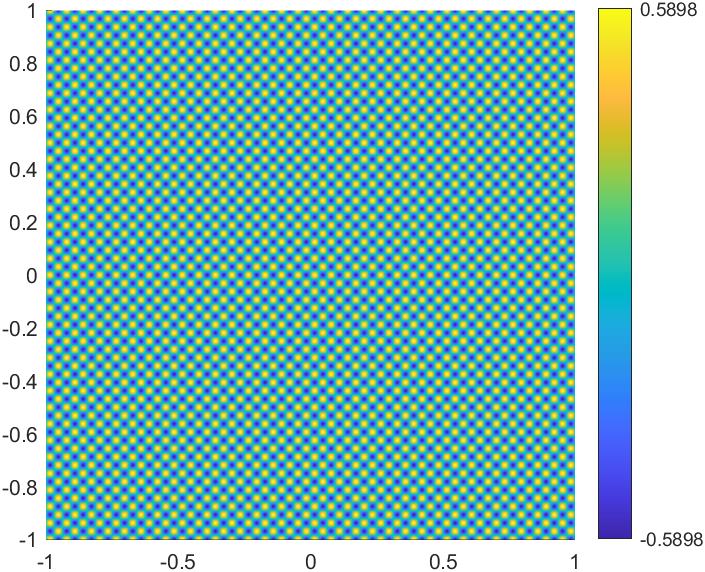}
		\subcaption{$b^{(3)}_{1}$}
		\label{subfig:combi_b1}
	\end{subfigure}
	\begin{subfigure}{0.32\textwidth}
		\includegraphics[width=\textwidth]{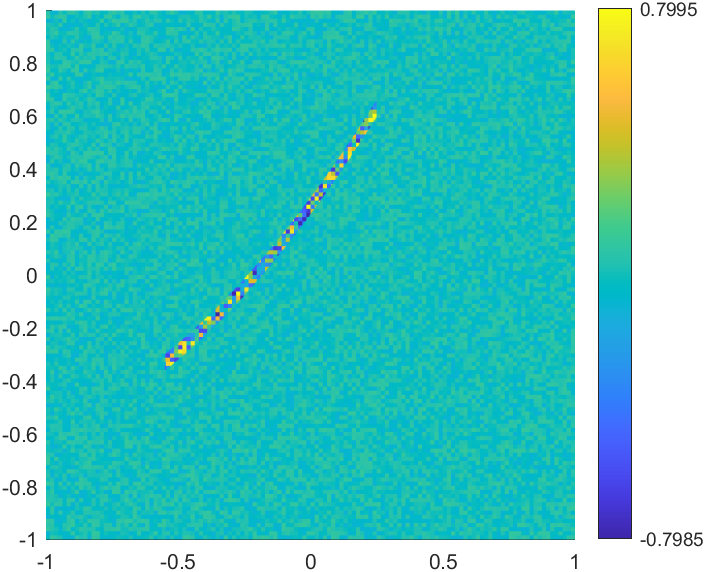}
		\subcaption{$b^{(3)}_{2}$}
		\label{subfig:combi_b2}
	\end{subfigure}
	\begin{subfigure}{0.32\textwidth}
		\includegraphics[width=\textwidth]{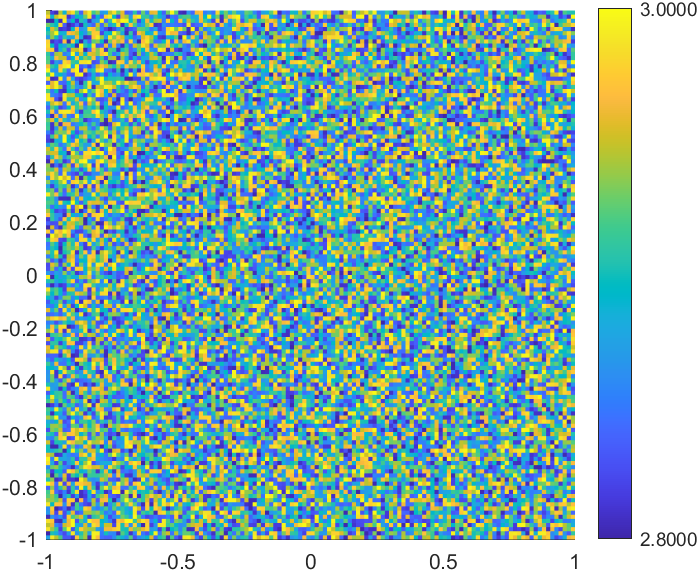}
		\subcaption{$c^{(3)}$}
		\label{subfig:combi_c}
	\end{subfigure}
	\caption{Illustration of the coefficients chosen in \cref{Subsec: combicoeff}.}
	\label{fig:coeff_combi}
\end{figure}

\begin{figure}
	\input{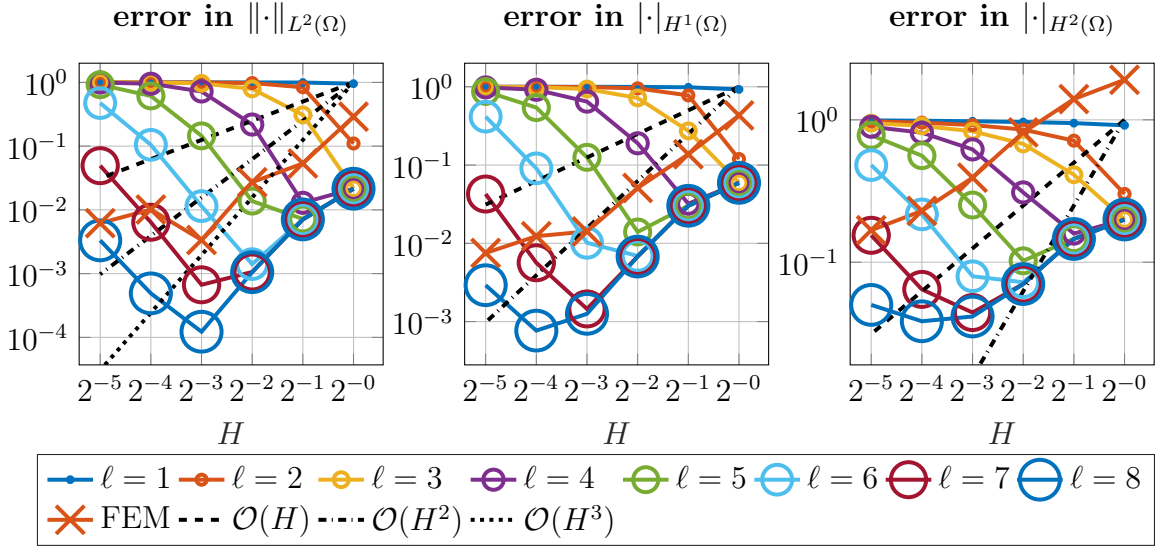}
	\caption{\small Relative errors for the combined problem ($A=A^{(3)}$, $b=b^{(3)}$, $c=c^{(3)}$, $f = f^{(3)}$) and $\lambda=1$.}
	\label{fig:conv_combi_lo}
\end{figure}

\subsection{Conclusions from the numerical experiments}

In our numerical experiments, we observe that the approximation errors for the LOD scheme in the $L^2$-norm, the $H^1$-seminorm and the $H^2$-seminorm stay below the corresponding errors for the classical $H^2$-conforming finite element method and decay at a faster rate with respect to $H$ (order 3/2/1 for LOD as compared to order 1-2/1-2/1 for classical FEM) for $\ell\approx \lvert\log(H^2)\rvert$. The numerical homogenization scheme further avoids the pre-asymptotic behavior of the classical $H^2$-conforming finite element scheme in the $H^2$-norm. Let us note that the increase of the approximation error of the LOD method for fixed $\ell$ and small $H$ is as expected and that it can be explained with the negative powers of $H$ in the error bound from Theorem \ref{Thm: Analysis LOD}(iii). As compared to the result of Theorem \ref{Thm: Analysis LOD}, the numerical evidence suggests that the proposed method also converges in $H^2$ and that an improved $H^1$-error bound might be possible to achieve for right-hand sides $f\in H^1(\Omega)$. The results from \cref{Subsec: combicoeff} indicate that the $\calO(H)$ term in our estimate from Theorem \ref{Thm: Analysis LOD}(iii) is sharp for $f\in L^2(\Omega)\backslash H^1(\Omega)$.

\section{Alternative discretization using a mixed FEM}\label{sec:Alternative discretization using a mixed non-symmetric FEM}

We emphasize that many other discretizations and combinations with quantities of interest are possible. The method from \cref{sec:Numerical homogenization scheme} is one particular example and represents the proof of concept that numerical homogenization for nondivergence-form PDEs is possible. During the work on this project many other possible discretizations were at hand and we want to present one particularly simple example, which is also quite efficient in the computations and allows using standard finite element techniques. The idea is to use a mixed formulation introduced in \cite{Gal17a,Gal17b} and improved in \cite{Gal19}. For the detailed derivation of the scheme we refer to the respective references. Hence, in this section, we use the mixed method for the solution of the global problem \cref{eq:weak} and the local problems as given in \cref{rem:saddlepoint} subject to a uniformly refined mesh $\calT_h$ that resolves all fine scales. It should be emphasized that the admissible right-hand
sides when using a mixed system instead of the problem from \cref{rem:saddlepoint}
reduce to a strict subspace of $V^{\ast}$.
In particular, point evaluations are unbounded functionals in
the mixed setting. Therefore, in our computations,
the quantities of interest $q_j$ in the problem from \cref{rem:saddlepoint} are
replaced by averaging operations, also known as quasi-interpolation.

We briefly illustrate the performance of the mixed method for the coefficients from the previous section with vanishing lower-order terms. The results can be found in \cref{fig:exp1_mixed,fig:exp2_cos_mixed,fig:exp3_cos_mixed}. In general, we observe that a significantly lower number of oversampling layers is sufficient to achieve similar errors. It is worth noting that in the case of $f\in L^2(\Omega)\backslash H^1(\Omega)$ the order reduction for the error in the $H^1$-seminorm does not seem to be present. Note that a full error analysis of this scheme is not contained in this work but follows along the same lines of what was presented in this work, with modifications for the mixed formulation and the above-mentioned quasi interpolation.

\begin{figure}
%
%
\definecolor{mycolor1}{rgb}{0.00000,0.44700,0.74100}%
\definecolor{mycolor2}{rgb}{0.85000,0.32500,0.09800}%
\begin{tikzpicture}

\begin{axis}[%
width=4cm,
height=4cm,
at={(0cm,0cm)},
scale only axis,
xmode=log,
xmin=0.0234375,
xmax=1.5,
xtick={0.03125,0.0625,0.125,0.25,0.5,1},
xticklabels={{$2^{-5}$},{$2^{-4}$},{$2^{-3}$},{$2^{-2}$},{$2^{-1}$},{$2^{-0}$}},
xminorticks=true,
xlabel style={font=\color{white!15!black}},
xlabel={$H$},
ymode=log,
ymin=1e-04+0.6*(1e-05-1e-04),
ymax=2,
yminorticks=false,
axis background/.style={fill=white},
title style={font=\bfseries},
title={error in $\lVert \cdot \rVert_{L^2(\Omega)}$},
xmajorgrids,
ymajorgrids,
legend style={at={(1.7,-0.5)}, anchor=south, legend columns=5, legend cell align=left, align=left, draw=white!15!black}
]
\addplot [color=mycolor1, line width=1.5pt, mark size=7.5pt, mark=o, mark options={solid, mycolor1}]
  table[row sep=crcr]{%
1	0.641451176073436\\
0.5	0.286988368017015\\
0.25	0.0388455627331714\\
0.125	0.00838471877042072\\
0.0625	0.00118806643441146\\
0.03125	0.000115336433574346\\
};
\addlegendentry{$\ell = \left\lceil \left|\ln(\tfrac{H}{2})\right|\right\rceil$}

\addplot [color=mycolor2, line width=1.5pt, mark size=7.5pt, mark=x, mark options={solid, mycolor2}]
  table[row sep=crcr]{%
1	0.97381867819688\\
0.5	0.8586675082118\\
0.25	0.424153017667632\\
0.125	0.262818117897754\\
0.0625	0.270227833680969\\
0.03125	0.283002128405462\\
};
\addlegendentry{FEM}

\addplot [color=black, dashed, line width=1.5pt]
  table[row sep=crcr]{%
1	0.641451176073436\\
0.03125	0.0200453492522949\\
};
\addlegendentry{$\mathcal{O}(H)$}

\addplot [color=black, dashdotted, line width=1.5pt]
  table[row sep=crcr]{%
1	0.641451176073436\\
0.03125	0.000626417164134215\\
};
\addlegendentry{$\mathcal{O}(H^2)$}

\addplot [color=black, dotted, line width=1.5pt]
  table[row sep=crcr]{%
1	0.641451176073436\\
0.03125	1.95755363791942e-05\\
};
\addlegendentry{$\mathcal{O}(H^3)$}

\end{axis}

\begin{axis}[%
width=4cm,
height=4cm,
at={(5.072cm,0cm)},
scale only axis,
xmode=log,
xmin=0.0234375,
xmax=1.5,
xtick={0.03125,0.0625,0.125,0.25,0.5,1},
xticklabels={{$2^{-5}$},{$2^{-4}$},{$2^{-3}$},{$2^{-2}$},{$2^{-1}$},{$2^{-0}$}},
xminorticks=true,
xlabel style={font=\color{white!15!black}},
xlabel={$H$},
ymode=log,
ymin=1e-02+0.8*(1e-03-1e-02),
ymax=2,
yminorticks=false,
axis background/.style={fill=white},
title style={font=\bfseries},
title={error in $\lvert \cdot \rvert_{H^1(\Omega)}$},
xmajorgrids,
ymajorgrids
]
\addplot [color=mycolor1, line width=1.5pt, mark size=7.5pt, mark=o, mark options={solid, mycolor1}, forget plot]
  table[row sep=crcr]{%
1	0.891883629579547\\
0.5	0.644861335301735\\
0.25	0.261811739282175\\
0.125	0.0843445291245769\\
0.0625	0.0231417030417492\\
0.03125	0.00623295568000836\\
};
\addplot [color=mycolor2, line width=1.5pt, mark size=7.5pt, mark=x, mark options={solid, mycolor2}, forget plot]
  table[row sep=crcr]{%
1	1.2516540031768\\
0.5	0.951162735846531\\
0.25	0.473878057195017\\
0.125	0.215085369844304\\
0.0625	0.167239830930142\\
0.03125	0.168988305945751\\
};
\addplot [color=black, dashed, line width=1.5pt, forget plot]
  table[row sep=crcr]{%
1	0.891883629579547\\
0.03125	0.0278713634243609\\
};
\addplot [color=black, dashdotted, line width=1.5pt, forget plot]
  table[row sep=crcr]{%
1	0.891883629579547\\
0.03125	0.000870980107011277\\
};
\end{axis}

\begin{axis}[%
width=4cm,
height=4cm,
at={(10.145cm,0cm)},
scale only axis,
xmode=log,
xmin=0.0234375,
xmax=1.5,
xtick={0.03125,0.0625,0.125,0.25,0.5,1},
xticklabels={{$2^{-5}$},{$2^{-4}$},{$2^{-3}$},{$2^{-2}$},{$2^{-1}$},{$2^{-0}$}},
xminorticks=true,
xlabel style={font=\color{white!15!black}},
xlabel={$H$},
ymode=log,
ymin=1e-01+0.6*(1e-02-1e-01),
ymax=1.5,
yminorticks=false,
axis background/.style={fill=white},
title style={font=\bfseries},
title={error in $\lvert \cdot \rvert_{H^2(\Omega)}$},
xmajorgrids,
ymajorgrids
]
\addplot [color=mycolor1, line width=1.5pt, mark size=7.5pt, mark=o, mark options={solid, mycolor1}, forget plot]
  table[row sep=crcr]{%
1	0.92045919381489\\
0.5	0.738538503849554\\
0.25	0.450837347332483\\
0.125	0.246901192266358\\
0.0625	0.129788760482899\\
0.03125	0.0713577264016259\\
};
\addplot [color=mycolor2, line width=1.5pt, mark size=7.5pt, mark=x, mark options={solid, mycolor2}, forget plot]
  table[row sep=crcr]{%
1	1.03879956933424\\
0.5	0.975311349298953\\
0.25	0.804254768487447\\
0.125	0.569605198540096\\
0.0625	0.411960144502146\\
0.03125	0.344630482837277\\
};
\addplot [color=black, dashed, line width=1.5pt, forget plot]
  table[row sep=crcr]{%
1	0.92045919381489\\
0.03125	0.0287643498067153\\
};
\end{axis}
\end{tikzpicture}%
	\caption{\small Relative errors for the periodic problem with vanishing lower-order terms ($A = A^{(1)}$, $b = 0$, $c = 0$, $f = f^{(1)}$) using the mixed method.}
	\label{fig:exp1_mixed}
\end{figure}
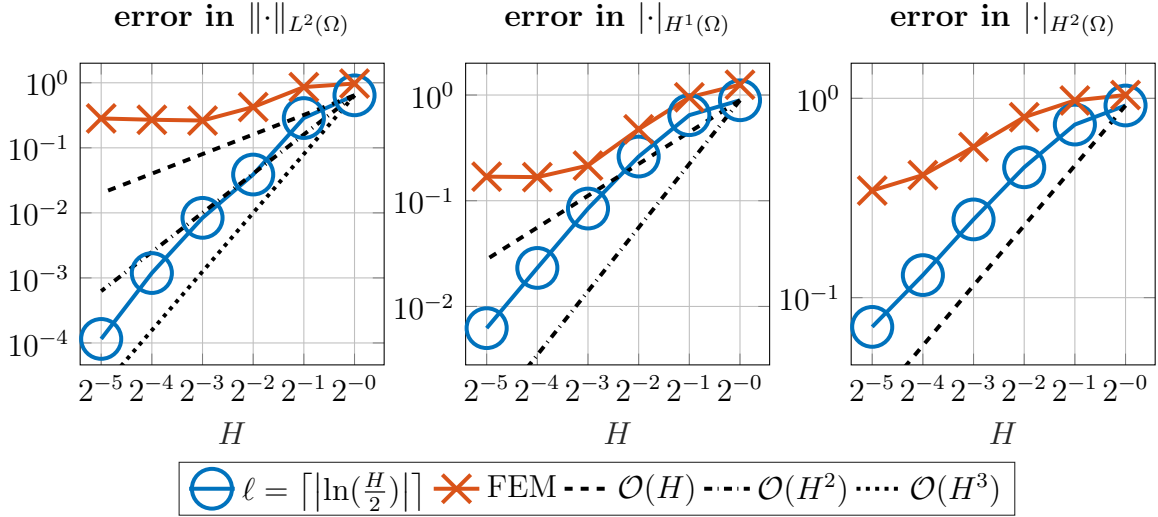

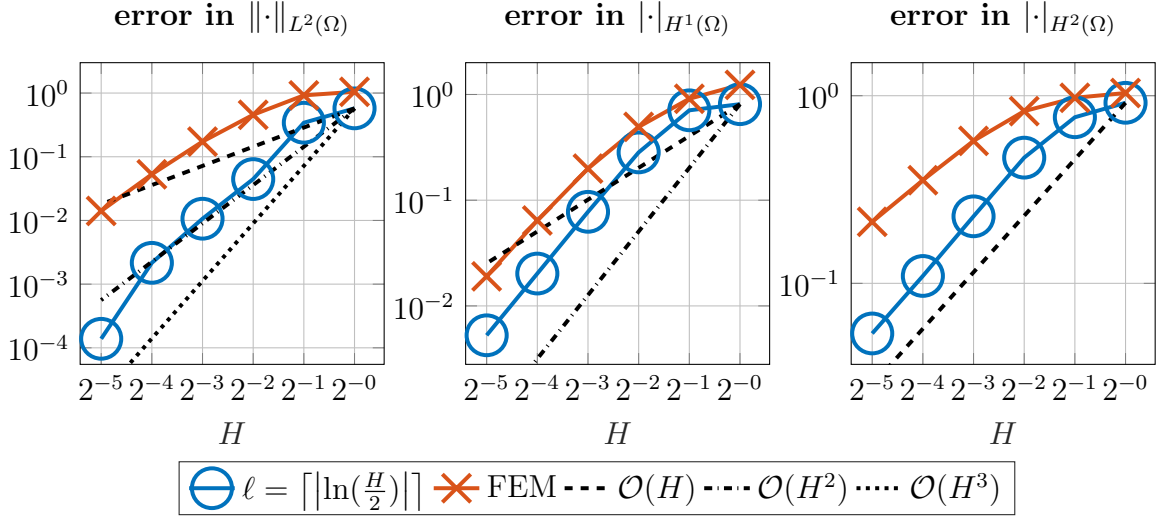
\begin{figure}
%
%
\definecolor{mycolor1}{rgb}{0.00000,0.44700,0.74100}%
\definecolor{mycolor2}{rgb}{0.85000,0.32500,0.09800}%
\begin{tikzpicture}

\begin{axis}[%
width=4cm,
height=4cm,
at={(0cm,0cm)},
scale only axis,
xmode=log,
xmin=0.0234375,
xmax=1.5,
xtick={0.03125,0.0625,0.125,0.25,0.5,1},
xticklabels={{$2^{-5}$},{$2^{-4}$},{$2^{-3}$},{$2^{-2}$},{$2^{-1}$},{$2^{-0}$}},
xminorticks=true,
xlabel style={font=\color{white!15!black}},
xlabel={$H$},
ymode=log,
ymin=1e-04+0.5*(1e-05-1e-04),
ymax=3,
yminorticks=false,
axis background/.style={fill=white},
title style={font=\bfseries},
title={error in $\lVert \cdot \rVert_{L^2(\Omega)}$},
xmajorgrids,
ymajorgrids,
legend style={at={(1.7,-0.5)}, anchor=south, legend columns=5, legend cell align=left, align=left, draw=white!15!black}
]
\addplot [color=mycolor1, line width=1.5pt, mark size=7.5pt, mark=o, mark options={solid, mycolor1}]
  table[row sep=crcr]{%
1	0.574985104935603\\
0.5	0.342179063560794\\
0.25	0.0446276001089037\\
0.125	0.0106433555431832\\
0.0625	0.00216068397881719\\
0.03125	0.000138596778743139\\
};
\addlegendentry{$\ell = \left\lceil \left|\ln(\tfrac{H}{2})\right| \right\rceil$}

\addplot [color=mycolor2, line width=1.5pt, mark size=7.5pt, mark=x, mark options={solid, mycolor2}]
  table[row sep=crcr]{%
1	1.0463093093076\\
0.5	0.920269033372505\\
0.25	0.456387840818945\\
0.125	0.172974087588458\\
0.0625	0.0529969301746803\\
0.03125	0.0141092315775703\\
};
\addlegendentry{FEM}

\addplot [color=black, dashed, line width=1.5pt]
  table[row sep=crcr]{%
1	0.574985104935602\\
0.03125	0.0179682845292376\\
};
\addlegendentry{$\mathcal{O}(H)$}

\addplot [color=black, dashdotted, line width=1.5pt]
  table[row sep=crcr]{%
1	0.574985104935603\\
0.03125	0.000561508891538674\\
};
\addlegendentry{$\mathcal{O}(H^2)$}

\addplot [color=black, dotted, line width=1.5pt]
  table[row sep=crcr]{%
1	0.574985104935602\\
0.03125	1.75471528605836e-05\\
};
\addlegendentry{$\mathcal{O}(H^3)$}

\end{axis}

\begin{axis}[%
width=4cm,
height=4cm,
at={(5.072cm,0cm)},
scale only axis,
xmode=log,
xmin=0.0234375,
xmax=1.5,
xtick={0.03125,0.0625,0.125,0.25,0.5,1},
xticklabels={{$2^{-5}$},{$2^{-4}$},{$2^{-3}$},{$2^{-2}$},{$2^{-1}$},{$2^{-0}$}},
xminorticks=true,
xlabel style={font=\color{white!15!black}},
xlabel={$H$},
ymode=log,
ymin=1e-02+0.8*(1e-03-1e-02),
ymax=2,
yminorticks=false,
axis background/.style={fill=white},
title style={font=\bfseries},
title={error in $\lvert \cdot \rvert_{H^1(\Omega)}$},
xmajorgrids,
ymajorgrids
]
\addplot [color=mycolor1, line width=1.5pt, mark size=7.5pt, mark=o, mark options={solid, mycolor1}, forget plot]
  table[row sep=crcr]{%
1	0.809542219283254\\
0.5	0.708460650927995\\
0.25	0.283272205315574\\
0.125	0.077664782151968\\
0.0625	0.0202355310541372\\
0.03125	0.00527368938170682\\
};
\addplot [color=mycolor2, line width=1.5pt, mark size=7.5pt, mark=x, mark options={solid, mycolor2}, forget plot]
  table[row sep=crcr]{%
1	1.23288451928682\\
0.5	0.911016399342147\\
0.25	0.494032680626576\\
0.125	0.199075971383865\\
0.0625	0.0642903512524902\\
0.03125	0.0190385974882474\\
};
\addplot [color=black, dashed, line width=1.5pt, forget plot]
  table[row sep=crcr]{%
1	0.809542219283254\\
0.03125	0.0252981943526017\\
};
\addplot [color=black, dashdotted, line width=1.5pt, forget plot]
  table[row sep=crcr]{%
1	0.809542219283254\\
0.03125	0.000790568573518802\\
};
\end{axis}

\begin{axis}[%
width=4cm,
height=4cm,
at={(10.145cm,0cm)},
scale only axis,
xmode=log,
xmin=0.0234375,
xmax=1.5,
xtick={0.03125,0.0625,0.125,0.25,0.5,1},
xticklabels={{$2^{-5}$},{$2^{-4}$},{$2^{-3}$},{$2^{-2}$},{$2^{-1}$},{$2^{-0}$}},
xminorticks=true,
xlabel style={font=\color{white!15!black}},
xlabel={$H$},
ymode=log,
ymin=1e-01+0.7*(1e-02-1e-01),
ymax=1.5,
yminorticks=false,
axis background/.style={fill=white},
title style={font=\bfseries},
title={error in $\lvert \cdot \rvert_{H^2(\Omega)}$},
xmajorgrids,
ymajorgrids
]
\addplot [color=mycolor1, line width=1.5pt, mark size=7.5pt, mark=o, mark options={solid, mycolor1}, forget plot]
  table[row sep=crcr]{%
1	0.916839262686384\\
0.5	0.76788642560638\\
0.25	0.467637485081934\\
0.125	0.227606153386078\\
0.0625	0.109489327822641\\
0.03125	0.0540586939427825\\
};
\addplot [color=mycolor2, line width=1.5pt, mark size=7.5pt, mark=x, mark options={solid, mycolor2}, forget plot]
  table[row sep=crcr]{%
1	1.03367433570522\\
0.5	0.974891347802894\\
0.25	0.827608253288689\\
0.125	0.574513643055381\\
0.0625	0.354548193462495\\
0.03125	0.212785734401188\\
};
\addplot [color=black, dashed, line width=1.5pt, forget plot]
  table[row sep=crcr]{%
1	0.916839262686384\\
0.03125	0.0286512269589495\\
};
\end{axis}
\end{tikzpicture}%
	\caption{\small Relative errors for the crack problem with vanishing lower-order terms ($A = A^{(2)}$, $b = 0$, $c = 0$, $f = f^{(2)}$) using the mixed method.}
	\label{fig:exp2_cos_mixed}
\end{figure}

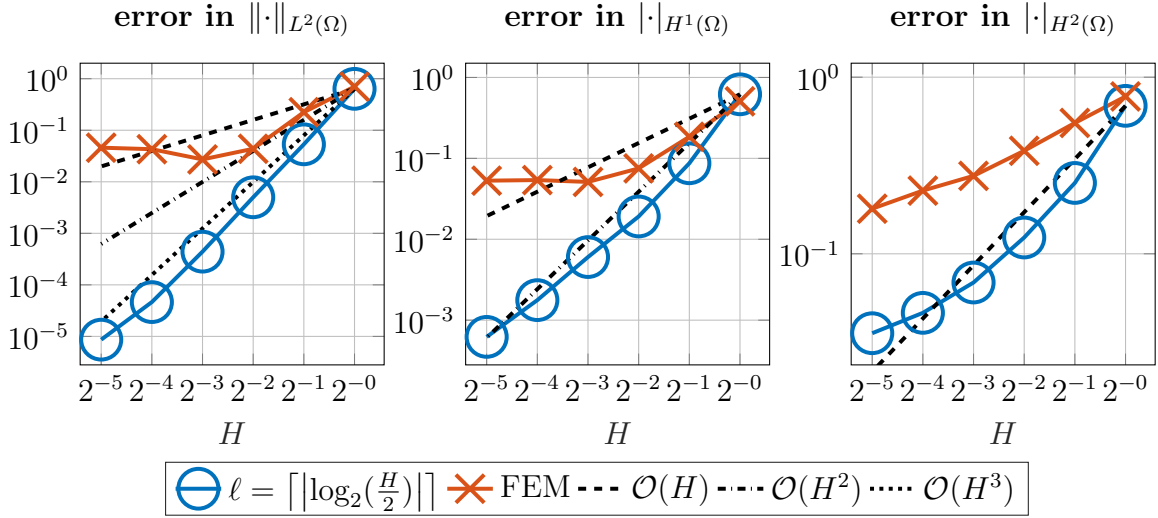
\begin{figure}
%
%
\definecolor{mycolor1}{rgb}{0.00000,0.44700,0.74100}%
\definecolor{mycolor2}{rgb}{0.85000,0.32500,0.09800}%
\begin{tikzpicture}

\begin{axis}[%
width=4cm,
height=4cm,
at={(0cm,0cm)},
scale only axis,
xmode=log,
xmin=0.0234375,
xmax=1.5,
xtick={0.03125,0.0625,0.125,0.25,0.5,1},
xticklabels={{$2^{-5}$},{$2^{-4}$},{$2^{-3}$},{$2^{-2}$},{$2^{-1}$},{$2^{-0}$}},
xminorticks=true,
xlabel style={font=\color{white!15!black}},
xlabel={$H$},
ymode=log,
ymin=1e-05+0.8*(1e-06-1e-05),
ymax=2,
yminorticks=false,
axis background/.style={fill=white},
title style={font=\bfseries},
title={error in $\lVert \cdot \rVert_{L^2(\Omega)}$},
xmajorgrids,
ymajorgrids,
legend style={at={(1.7,-0.5)}, anchor=south, legend columns=5, legend cell align=left, align=left, draw=white!15!black}
]
\addplot [color=mycolor1, line width=1.5pt, mark size=7.5pt, mark=o, mark options={solid, mycolor1}]
  table[row sep=crcr]{%
1	0.635891851596285\\
0.5	0.0533169510113404\\
0.25	0.00500152091333216\\
0.125	0.00044099507943011\\
0.0625	4.63067222202758e-05\\
0.03125	8.64027391420273e-06\\
};
\addlegendentry{$\ell = \left\lceil \left|\log_2(\tfrac{H}{2})\right| \right\rceil$}

\addplot [color=mycolor2, line width=1.5pt, mark size=7.5pt, mark=x, mark options={solid, mycolor2}]
  table[row sep=crcr]{%
1	0.703941406150318\\
0.5	0.224106824855199\\
0.25	0.0437110274348503\\
0.125	0.0271724915366338\\
0.0625	0.042901248291434\\
0.03125	0.0459014489931656\\
};
\addlegendentry{FEM}

\addplot [color=black, dashed, line width=1.5pt]
  table[row sep=crcr]{%
1	0.635891851596285\\
0.03125	0.0198716203623839\\
};
\addlegendentry{$\mathcal{O}(H)$}

\addplot [color=black, dashdotted, line width=1.5pt]
  table[row sep=crcr]{%
1	0.635891851596285\\
0.03125	0.000620988136324497\\
};
\addlegendentry{$\mathcal{O}(H^2)$}

\addplot [color=black, dotted, line width=1.5pt]
  table[row sep=crcr]{%
1	0.635891851596285\\
0.03125	1.94058792601405e-05\\
};
\addlegendentry{$\mathcal{O}(H^3)$}

\end{axis}

\begin{axis}[%
width=4cm,
height=4cm,
at={(5.072cm,0cm)},
scale only axis,
xmode=log,
xmin=0.0234375,
xmax=1.5,
xtick={0.03125,0.0625,0.125,0.25,0.5,1},
xticklabels={{$2^{-5}$},{$2^{-4}$},{$2^{-3}$},{$2^{-2}$},{$2^{-1}$},{$2^{-0}$}},
xminorticks=true,
xlabel style={font=\color{white!15!black}},
xlabel={$H$},
ymode=log,
ymin=1e-03+0.8*(1e-04-1e-03),
ymax=1.5,
yminorticks=false,
axis background/.style={fill=white},
title style={font=\bfseries},
title={error in $\lvert \cdot \rvert_{H^1(\Omega)}$},
xmajorgrids,
ymajorgrids
]
\addplot [color=mycolor1, line width=1.5pt, mark size=7.5pt, mark=o, mark options={solid, mycolor1}, forget plot]
  table[row sep=crcr]{%
1	0.618849076147286\\
0.5	0.0869351106273472\\
0.25	0.0191312552034705\\
0.125	0.00602477119779293\\
0.0625	0.00176806220699643\\
0.03125	0.000621343839093012\\
};
\addplot [color=mycolor2, line width=1.5pt, mark size=7.5pt, mark=x, mark options={solid, mycolor2}, forget plot]
  table[row sep=crcr]{%
1	0.506780833834574\\
0.5	0.184899088117734\\
0.25	0.0751604454173286\\
0.125	0.0511220187805707\\
0.0625	0.053883576565829\\
0.03125	0.0529864877223165\\
};
\addplot [color=black, dashed, line width=1.5pt, forget plot]
  table[row sep=crcr]{%
1	0.618849076147285\\
0.03125	0.0193390336296027\\
};
\addplot [color=black, dashdotted, line width=1.5pt, forget plot]
  table[row sep=crcr]{%
1	0.618849076147286\\
0.03125	0.000604344800925083\\
};
\end{axis}

\begin{axis}[%
width=4cm,
height=4cm,
at={(10.145cm,0cm)},
scale only axis,
xmode=log,
xmin=0.0234375,
xmax=1.5,
xtick={0.03125,0.0625,0.125,0.25,0.5,1},
xticklabels={{$2^{-5}$},{$2^{-4}$},{$2^{-3}$},{$2^{-2}$},{$2^{-1}$},{$2^{-0}$}},
xminorticks=true,
xlabel style={font=\color{white!15!black}},
xlabel={$H$},
ymode=log,
ymin=1e-01+0.85*(1e-02-1e-01),
ymax=1.2,
yminorticks=false,
axis background/.style={fill=white},
title style={font=\bfseries},
title={error in $\lvert \cdot \rvert_{H^2(\Omega)}$},
xmajorgrids,
ymajorgrids
]
\addplot [color=mycolor1, line width=1.5pt, mark size=7.5pt, mark=o, mark options={solid, mycolor1}, forget plot]
  table[row sep=crcr]{%
1	0.688260793610122\\
0.5	0.251877331119953\\
0.25	0.123687165347801\\
0.125	0.06878370405649\\
0.0625	0.046305254909197\\
0.03125	0.0355146505314477\\
};
\addplot [color=mycolor2, line width=1.5pt, mark size=7.5pt, mark=x, mark options={solid, mycolor2}, forget plot]
  table[row sep=crcr]{%
1	0.778032255576848\\
0.5	0.550703484415861\\
0.25	0.383183898990304\\
0.125	0.276259711603375\\
0.0625	0.227943452498199\\
0.03125	0.180083270308305\\
};
\addplot [color=black, dashed, line width=1.5pt, forget plot]
  table[row sep=crcr]{%
1	0.688260793610122\\
0.03125	0.0215081498003163\\
};
\end{axis}
\end{tikzpicture}%
	\caption{\small Relative errors for the combined problem with vanishing lower-order terms ($A=A^{(3)}$, $b=0$, $c=0$, $f = f^{(3)}$) using the mixed method.}
	\label{fig:exp3_cos_mixed}
\end{figure}

\section{Concluding remarks}

\subsection{Review} 
In this work, we presented a novel numerical homogenization scheme for linear second-order elliptic PDEs in nondivergence-form with coefficients that satisfy a (generalized) Cordes condition. 
Motivated by the degrees of freedom of the nonconforming Morley finite element, our approach using the LOD framework provides a proof of concept that this method is also applicable to the class of nondivergence-form PDEs. The error analysis revealed that numerical homogenization is applicable to problems with coefficients that do not exhibit any scale separation, even beyond periodicity. Moreover, the favorable accuracy properties of the classical LOD for divergence-form PDEs are preserved. Various numerical experiments have been performed that support the theoretical findings.

\subsection{Extensions and future work}\label{Sec: ext and fw} 

Finally, we give some remarks regarding extensions of our results and we address future work.

\subsubsection{The case $n\geq 3$} 

In \cref{sec:Numerical homogenization scheme}, we assumed that $n = 2$ for simplicity of the presentation of the methodology. It is straightforward to adapt this to dimensions $n\geq 3$ by defining the quantities of interest corresponding to the degrees of freedom of the Morley element in dimension $n$; see \cite{WX13}. 

\subsubsection{$H^2$-convergence of the numerical homogenization scheme}

The numerical experiments suggest that the numerical homogenization scheme presented in this work converges not only in the $H^1$-norm, but also in the $H^2$-norm. A proof of an $H^2$-error bound is subject of future work.

\subsubsection{Improved localization}

We emphasize that, using an improved localization technique proposed in \cite{HP22}, increasing errors for refinements in $H$ for fixed $\ell$ can be cured. A potential application of the improved localization using the Super-Localized Orthogonal Decomposition \cite{HP21pre,FHP21,BFP22} might also be possible. Moreover, we see the potential to use the proposed scheme as a preconditioner.

\subsubsection{Different problem classes}

The method presented in this paper can be applied to any Lax--Milgram-type problem over $V = H^2(\Omega)\cap H^1_0(\Omega)$ of the form \eqref{intro3} with $F\in V^{\ast}$ and a locally bounded and coercive bilinear form $a:V\times V\rightarrow \R$.

\section*{Acknowledgments}
The work of P. Freese and D. Gallistl is part of projects that have received funding from the European Research Council ERC under the European Union's Horizon 2020 research and innovation program (Grant agreements No.~865751 and 
No.~891734). D. Peterseim acknowledges funding by the Deutsche Forschungsgemeinschaft within the research project \emph{Convexified Variational Formulations at Finite Strains based on Homogenised Damaged Microstructures} (PE 2143/5-1).

\bibliographystyle{plain}
\bibliography{ref_LOD_nondiv}

\begin{thebibliography}{10}

\bibitem{AHP21}
R.~Altmann, P.~Henning, and D.~Peterseim.
\newblock Numerical homogenization beyond scale separation.
\newblock {\em Acta Numer.}, 30:1--86, 2021.

\bibitem{AK18}
D.~Arjmand and G.~Kreiss.
\newblock An equation-free approach for second order multiscale hyperbolic
  problems in non-divergence form.
\newblock {\em Commun. Math. Sci.}, 16(8):2317--2343, 2018.

\bibitem{AL89}
M.~Avellaneda and F.-H. Lin.
\newblock Compactness methods in the theory of homogenization. {II}.
  {E}quations in nondivergence form.
\newblock {\em Comm. Pure Appl. Math.}, 42(2):139--172, 1989.

\bibitem{BaL11}
I.~Babu{\v s}ka and R.~Lipton.
\newblock Optimal local approximation spaces for generalized finite element
  methods with application to multiscale problems.
\newblock {\em Multiscale Model. Simul.}, 9(1):373--406, 2011.

\bibitem{BLP11}
A.~Bensoussan, J.-L. Lions, and G.~Papanicolaou.
\newblock {\em Asymptotic analysis for periodic structures}.
\newblock AMS Chelsea Publishing, Providence, RI, 2011.
\newblock Corrected reprint of the 1978 original.

\bibitem{BM74}
G.~Birkhoff and L.~Mansfield.
\newblock Compatible triangular finite elements.
\newblock {\em J. Math. Anal. Appl.}, 47:531--553, 1974.

\bibitem{BFP22}
F.~Bonizzoni, P.~Freese, and D.~Peterseim.
\newblock Super-localized orthogonal decomposition for convection-dominated
  diffusion problems, 2022.
\newblock arXiv:2206.01975.

\bibitem{CM09}
F.~Camilli and C.~Marchi.
\newblock Rates of convergence in periodic homogenization of fully nonlinear
  uniformly elliptic {PDE}s.
\newblock {\em Nonlinearity}, 22(6):1481--1498, 2009.

\bibitem{CSS20}
Y.~Capdeboscq, T.~Sprekeler, and E.~S\"{u}li.
\newblock Finite element approximation of elliptic homogenization problems in
  nondivergence-form.
\newblock {\em ESAIM Math. Model. Numer. Anal.}, 54(4):1221--1257, 2020.

\bibitem{Cia78}
P.~G. Ciarlet.
\newblock {\em The finite element method for elliptic problems}.
\newblock Studies in Mathematics and its Applications, Vol. 4. North-Holland
  Publishing Co., Amsterdam-New York-Oxford, 1978.

\bibitem{EGH13}
Y.~Efendiev, J.~Galvis, and T.~Y. Hou.
\newblock Generalized multiscale finite element methods ({GM}s{FEM}).
\newblock {\em J. Comput. Physics}, 251:116--135, 2013.

\bibitem{FO18}
C.~Finlay and A.~M. Oberman.
\newblock Approximate homogenization of convex nonlinear elliptic {PDE}s.
\newblock {\em Commun. Math. Sci.}, 16(7):1895--1906, 2018.

\bibitem{FOO18}
C.~Finlay and A.~M. Oberman.
\newblock Approximate homogenization of fully nonlinear elliptic {PDE}s:
  estimates and numerical results for {P}ucci type equations.
\newblock {\em J. Sci. Comput.}, 77(2):936--949, 2018.

\bibitem{FHP21}
P.~Freese, M.~Hauck, and D.~Peterseim.
\newblock Super-localized orthogonal decomposition for high-frequency
  {H}elmholtz problems, 2021.
\newblock arXiv:2112.11368.

\bibitem{FO09}
B.~D. Froese and A.~M. Oberman.
\newblock Numerical averaging of non-divergence structure elliptic operators.
\newblock {\em Commun. Math. Sci.}, 7(4):785--804, 2009.

\bibitem{Gal15}
D.~Gallistl.
\newblock Morley finite element method for the eigenvalues of the biharmonic
  operator.
\newblock {\em IMA J. Numer. Anal.}, 35(4):1779--1811, 2015.

\bibitem{Gal17a}
D.~Gallistl.
\newblock Stable splitting of polyharmonic operators by generalized {S}tokes
  systems.
\newblock {\em Math. Comp.}, 86(308):2555--2577, 2017.

\bibitem{Gal17b}
D.~Gallistl.
\newblock Variational formulation and numerical analysis of linear elliptic
  equations in nondivergence form with {C}ordes coefficients.
\newblock {\em SIAM J. Numer. Anal.}, 55(2):737--757, 2017.

\bibitem{Gal19}
D.~Gallistl.
\newblock Numerical approximation of planar oblique derivative problems in
  nondivergence form.
\newblock {\em Math. Comp.}, 88(317):1091--1119, 2019.

\bibitem{GSS21}
D.~Gallistl, T.~Sprekeler, and E.~S\"{u}li.
\newblock Mixed {F}inite {E}lement {A}pproximation of {P}eriodic
  {H}amilton--{J}acobi--{B}ellman {P}roblems {W}ith {A}pplication to
  {N}umerical {H}omogenization.
\newblock {\em Multiscale Model. Simul.}, 19(2):1041--1065, 2021.

\bibitem{GST22}
X.~Guo, T.~Sprekeler, and H.~V. Tran.
\newblock Characterizations of diffusion matrices in homogenization of elliptic
  equations in nondivergence-form, 2022.
\newblock arXiv:2201.01974 [math.AP].

\bibitem{GTY20}
X.~Guo, H.~V. Tran, and Y.~Yu.
\newblock Remarks on optimal rates of convergence in periodic homogenization of
  linear elliptic equations in non-divergence form.
\newblock {\em SN Partial Differ. Equ. Appl.}, 1(15), 2020.

\bibitem{HP21pre}
M.~Hauck and D.~Peterseim.
\newblock Super-localization of elliptic multiscale problems, 2021.
\newblock arXiv:2107.13211.

\bibitem{HP22}
M.~Hauck and D.~Peterseim.
\newblock Multi-{R}esolution {L}ocalized {O}rthogonal {D}ecomposition for
  {H}elmholtz {P}roblems.
\newblock {\em Multiscale Model. Simul.}, 20(2):657--684, 2022.

\bibitem{HeP13}
P.~Henning and D.~Peterseim.
\newblock Oversampling for the multiscale finite element method.
\newblock {\em Multiscale Model. Simul.}, 11(4):1149--1175, 2013.

\bibitem{JKO94}
V.~V. Jikov, S.~M. Kozlov, and O.~A. Ole\u{\i}nik.
\newblock {\em Homogenization of differential operators and integral
  functionals}.
\newblock Springer-Verlag, Berlin, 1994.
\newblock Translated from the Russian.

\bibitem{KS22}
E.~L. Kawecki and T.~Sprekeler.
\newblock Discontinuous {G}alerkin and {C}0-{IP} finite element approximation
  of periodic {H}amilton-{J}acobi-{B}ellman-{I}saacs problems with application
  to numerical homogenization.
\newblock {\em ESAIM Math. Model. Numer. Anal.}, 56(2):679--704, 2022.

\bibitem{KL16}
S.~Kim and K.-A. Lee.
\newblock Higher order convergence rates in theory of homogenization: equations
  of non-divergence form.
\newblock {\em Arch. Ration. Mech. Anal.}, 219(3):1273--1304, 2016.

\bibitem{KPY18}
R.~Kornhuber, D.~Peterseim, and H.~Yserentant.
\newblock An analysis of a class of variational multiscale methods based on
  subspace decomposition.
\newblock {\em Math. Comp.}, 87(314):2765--2774, 2018.

\bibitem{Ma22}
C.~Ma, R.~Scheichl, and T.~Dodwell.
\newblock Novel design and analysis of generalized finite element methods based
  on locally optimal spectral approximations.
\newblock {\em SIAM J. Numer. Anal.}, 60(1):244--273, 2022.

\bibitem{Mai20}
R.~Maier.
\newblock A high-order approach to elliptic multiscale problems with general
  unstructured coefficients.
\newblock {\em SIAM J. Numer. Anal.}, 59(2):1067--1089, 2021.

\bibitem{MaP14}
A.~M{\aa}lqvist and D.~Peterseim.
\newblock Localization of elliptic multiscale problems.
\newblock {\em Math. Comp.}, 83(290):2583--2603, 2014.

\bibitem{MalP20}
A.~M{\aa}lqvist and D.~Peterseim.
\newblock {\em Numerical Homogenization by Localized Orthogonal Decomposition}.
\newblock Society for Industrial and Applied Mathematics, Philadelphia, PA,
  2020.

\bibitem{Owh17}
H.~Owhadi.
\newblock Multigrid with rough coefficients and multiresolution operator
  decomposition from hierarchical information games.
\newblock {\em SIAM Rev.}, 59(1):99--149, 2017.

\bibitem{OZ07}
H.~Owhadi and L.~Zhang.
\newblock Metric-based upscaling.
\newblock {\em Comm. Pure Appl. Math.}, 60(5):675--723, 2007.

\bibitem{OZB14}
H.~Owhadi, L.~Zhang, and L.~Berlyand.
\newblock Polyharmonic homogenization, rough polyharmonic splines and sparse
  super-localization.
\newblock {\em ESAIM Math. Model. Numer. Anal. (M2AN)}, 48(2):517--552, 2014.

\bibitem{SS13}
I.~Smears and E.~S\"{u}li.
\newblock Discontinuous {G}alerkin finite element approximation of
  nondivergence form elliptic equations with {C}ord\`es coefficients.
\newblock {\em SIAM J. Numer. Anal.}, 51(4):2088--2106, 2013.

\bibitem{SS14}
I.~Smears and E.~S\"{u}li.
\newblock Discontinuous {G}alerkin finite element approximation of
  {H}amilton-{J}acobi-{B}ellman equations with {C}ordes coefficients.
\newblock {\em SIAM J. Numer. Anal.}, 52(2):993--1016, 2014.

\bibitem{Spr23}
T.~Sprekeler.
\newblock Homogenization of nondivergence-form elliptic equations with
  discontinuous coefficients and finite element approximation of the
  homogenized problem, 2023.
\newblock arXiv:2305.19833 [math.NA].

\bibitem{ST21}
T.~Sprekeler and H.~V. Tran.
\newblock Optimal convergence rates for elliptic homogenization problems in
  nondivergence-form: analysis and numerical illustrations.
\newblock {\em Multiscale Model. Simul.}, 19(3):1453--1473, 2021.

\bibitem{Szy06}
D.~B. Szyld.
\newblock The many proofs of an identity on the norm of oblique projections.
\newblock {\em Numer. Algorithms}, 42(3-4):309--323, 2006.

\bibitem{WX13}
M.~Wang and J.~Xu.
\newblock Minimal finite element spaces for {$2m$}-th-order partial
  differential equations in {$R^n$}.
\newblock {\em Math. Comp.}, 82(281):25--43, 2013.

\end{thebibliography}

\end{document}